\documentclass[12pt]{article}
\usepackage[top=1in, bottom=1in, left=1in, right=1in]{geometry}

\usepackage{xcolor}
\usepackage{geometry}
\usepackage{amssymb,bm}
\usepackage{amsmath}
\usepackage{amsthm}
\usepackage{graphicx}
\usepackage{hyperref}
\usepackage{caption}
\usepackage{float} 
\hypersetup{pdfborder=0 0 0}

\newtheorem{theorem}{{\sc Theorem}}[section]

\newtheorem{lemma}[theorem]{{\sc Lemma}}
\newtheorem{corollary}[theorem]{Corollary}
\newtheorem{remark}[theorem]{Remark}
\newtheorem{definition}[theorem]{Definition}

\newcommand\restr[2]{{ \left.\kern-\nulldelimiterspace #1 \vphantom{\big|} \right|_{#2}}}
\newcommand{\RR}{\mathbb{R}}
\newcommand{\ZZ}{\mathbb{Z}}
\newcommand{\CC}{\mathbb{C}}

\newcommand{\n}{\noindent}
\newcommand{\comment}[1]{}
\renewcommand{\div}{\mathrm{div}}

\DeclareMathAlphabet{\pazocal}{OMS}{zplm}{m}{n}

\newcommand{\PC}{\pazocal{C}}
\newcommand{\CT}{\pazocal{T}}

\newcommand{\PI}{\pazocal{I}}

\newcommand{\CM}{\pazocal{M}}

\usepackage{bbm}
\usepackage{subcaption}

\title{Scattering from analytic and piecewise analytic inhomogeneities}
\author{Narek Hovsepyan\footnote{Department of Mathematics, Rutgers University, New Brunswick, NJ, USA (nh507@math.rutgers.edu)} \quad and \quad Michael S. Vogelius\footnote{Department of Mathematics, Rutgers University, New Brunswick, NJ, USA (vogelius@math.rutgers.edu)}}
\date{}

\begin{document}
\maketitle

\begin{abstract}
We study scattering for the linear Helmholtz operator in two dimensions and develop a technique, which can be used to ascertain scattering of a given incident wave from very regular inhomogeneities. This technique is then applied to a number of interesting examples.

\end{abstract}

\section{Introduction}
\setcounter{equation}{0}

The study of the precise effect of the geometry (shape) and regularity of an inhomogeneity on its ability to scatter all (or almost all) incident electromagnetic fields has received significant attention recently (see e.g. \cite{PSV,BPS,CV,SS21,VX,LHY,SS24} and references therein).

In this paper we focus on fields governed by the two dimensional linear Helmholtz equation $(\Delta +k^2n)v=0$, with the index of refraction, $n$, being $1$ outside the bounded inhomogeneity, and attaining a different constant value, $q$, inside. Within the class of Lipschitz inhomogeneities it has been shown that generically nonscattering can only occur if the inhomogeneity has a real analytic boundary \cite{CV}. The initial assumption on the class of inhomogeneities can be relaxed to allow for certain cusps \cite{SS21}. To be more precise, and specializing to our case of a constant refractive index $q \neq 1$, the results of \cite{CV, SS21} state that if the inhomogeneity $D$ has a boundary that is not real analytic at a point $x_0 \in \partial D$, then any incident wave that is nonvanishing at $x_0$ is scattered by $D$. This nonvanishing assumption on the incident wave was removed in the recent work \cite{SS25}, under more restrictive a priori assumptions on $D$. Specifically, the authors show that convex or piecewise $C^1$ planar inhomogeneities, whose boundary has a singular (non-$C^1$) point, always scatter. The case of piecewise $C^2$ planar inhomogeneities was treated in \cite{EH}. Analogous results also hold in higher dimensions for domains with edge singularities \cite{EH, SS25}.

For an inhomogeneity in the shape of a disk (or a ball) it is well known that there exists a countably infinite set of real wave numbers and associated Herglotz incident waves that are not scattered \cite{CK19,VX}. What is less understood are the scattering properties of inhomogeneities with real analytic boundaries other than spheres, as well as inhomogeneities with inward cusps (not excluded by the analysis in \cite{SS24}). A natural question is: ``does any simply connected inhomogeneity, with a non-spherical, real analytic  boundary and a constant index of refraction different from $1$, scatter any incident wave?". While we do not know the answer to this question, here we develop a technique that in many instances will allow one to find very minimal conditions on an incident wave that guarantee scattering. We apply this technique to a number of interesting examples, including the ellipse and a non-convex real analytic domain. We also apply it to two domains with cusps -- one where scattering was already affirmed in \cite{SS24} and one where the results in \cite{SS24} were inconclusive. Finally we apply it to the much studied example of a piecewise analytic domain with a corner -- confirming that this generically scatters.

We note that the existence of a (non-trivial) nonscattering incident field at wave number $k$ implies that $k$ is a so called transmission eigenvalue. It is well known that for any Lipschitz domain with a constant index of refraction $\ne 1$ there exists a countably infinite set of real transmission eigenvalues \cite{CGH}. The nonscattering question is therefore if that element of the eigenfunction pair, which represents the incident field, can be extended to a solution to $(\Delta +k^2)v=0$ in all of $\mathbb{R}^d$ for any of these transmission eigenvalues.

The nonscattering problem we study here is in some sense a more general version of the Schiffer (or Pompeiu) problem, and the techniques developed here are closely related to those used for the Pompeiu problem in \cite{GS91,GS91TR}. Our techniques are also related to those used in \cite{VX,VX2} to establish the finiteness of nonscattering wave numbers for non-circular domains in the case of  incident Herglotz waves with a fixed density.


\section{Preliminaries} \label{SECT Prelim}
\setcounter{equation}{0}

\n Consider a non-trivial incident wave at a fixed wave number $k>0$

\begin{equation} \label{u^in Helm}
\Delta u^{\text{in}} + k^2 u^{\text{in}} = 0 \qquad \qquad \text{in} \ \RR^2.
\end{equation}

\n Let $D$ be a bounded, simply connected planar domain whose boundary may include cusps and suppose that it has a constant refractive index $0 < q \neq 1$. Let $u^{\text{tot}}$ denote the total field in the presence of the inhomogeneity $D$ and the incident field $u^{\text{in}}$, and set $u=u^{\text{tot}}- u^{\text{in}}$. If $u^{\text{in}}$ is non scattering, then $u=0$ in $\mathbb{R}^2 \setminus D$ and solves

\begin{equation} \label{u system}
\Delta u + k^2 q u = k^2 (1 - q) u^{\text{in}}\chi_{D} \qquad
 \text{in} \ \mathbb{R}^2~,
\end{equation}
where $\chi_D$ denotes the indicator function of $D$. We readily conclude from \eqref{u system} that $u \in H^2(D)\cap C^{1}(D)$. 
Note that if $D$ is sufficiently regular, for instance Lipschitz or piecewise smooth with cusps, then the condition $u=0$ in $\mathbb{R}^2 \setminus D$ is equivalent to $u= \frac{\partial u}{\partial n}=0$ on $\Gamma$, the boundary of $D$, where $n$ denotes the unit outer normal. Multiplying the PDE by a test function $\phi$ and using the Green's identity we obtain the following necessary condition for  $u^{\text{in}}$ to be nonscattering

\begin{equation} \label{test fctn}
\int_D u^{\text{in}} \phi dx = 0 \qquad \qquad \text{for any} \quad \Delta \phi + k^2 q \phi = 0.
\end{equation}

\n Our goal is to derive a contradiction to the above equation by appropriately choosing a parametrized family of test functions and analyzing the asymptotic behavior of the above integral with respect to this parameter. The desired contradiction is achieved by imposing an appropriate nondegeneracy condition on $u^{\text{in}}$ which guarantees that the leading term (or the next-order term) of this asymptotic expansion is nonzero. By having reached a contradiction to \eqref{test fctn} this condition is thus sufficient to guarantee that $u^{\text{in}}$ is scattering.  However, analyzing the above integral is not an easy task. First, let us reduce it to an integral over the boundary $\Gamma$ using the equations that $u^{\text{in}}$ and $\phi$ satisfy. If we insert $u^{\text{in}} = -k^{-2} \Delta u^{\text{in}}$, and then use the Green's identity, we arrive at

\begin{equation*}
-k^2 \int_D u^{\text{in}} \phi dx = \int_D  u^{\text{in}}  \Delta \phi dx + \int_{\Gamma} \left(\partial_n u^{\text{in}} \phi - u^{\text{in}} \partial_n \phi\right) ds.
\end{equation*}

\n Replacing $\Delta \phi$ with $-k^2 q \phi$ and using that $q$ is constant, we deduce that 

\begin{equation*}
k^2 (q-1) \int_D u^{\text{in}} \phi dx = \int_{\Gamma} \left(\partial_n u^{\text{in}} \phi - u^{\text{in}} \partial_n \phi \right)ds.
\end{equation*}

\n Consequently, if $u^{\text{in}}$ does not scatter, then

\begin{equation} \label{test fctn 2}
\int_{\Gamma} \left( \partial_n u^{\text{in}} \phi - u^{\text{in}} \partial_n \phi \right) ds = 0 \qquad \qquad \text{for any} \quad \Delta \phi + k^2 q \phi = 0.
\end{equation}

\n The main idea is to take a plane wave solution as a test function, but with a complex wave vector. Namely, $\phi(x) = e^{i x \cdot \xi}$, where $\cdot$ denotes the inner product of $\RR^2$ and $\xi \in \CM$ with

\begin{equation} \label{CX}
\CM = \left\{ \xi \in \CC^2 : \xi \cdot \xi = k^2 q \right\}.
\end{equation}

\n The restriction imposed on $\xi$ ensures that $\phi$ satisfies the required PDE. For this choice of test function, \eqref{test fctn 2} can be simplified to

\begin{equation} \label{test fctn xi}
\int_{\Gamma} n \cdot \left( \nabla u^{\text{in}} - i \xi u^{\text{in}} \right) e^{i x \cdot \xi} ds = 0 \qquad \qquad \forall \, \xi \in \CM.
\end{equation}
The desired contradiction will be achieved by analyzing the asymptotics of the boundary integral above, as $|\Im \xi| \to \infty$. 

\n Note that $\xi \in \CM$ iff

\begin{equation*}
\Re \xi \cdot \Im \xi = 0 \qquad \text{and} \qquad
|\Re \xi|^2 - |\Im \xi|^2 = k^2 q.
\end{equation*}

\n Let us take the following family of vectors from $\CM$, parametrized by $\lambda \geq 0$:

\begin{equation*}
\Im \xi = \lambda \, (- 1, 0),
\qquad \qquad
\Re \xi = \sqrt{\lambda^2 + k^2 q} \, (0, 1).
\end{equation*}

\n It will be convenient to write

\begin{equation} \label{xi with lambda}
\xi = (- i \lambda, \sqrt{\lambda^2 + k^2 q}) = \lambda (-i, 1) + \tilde{\lambda} (0, 1),
\end{equation}

\n where

\begin{equation} \label{lambda tild asymp}
\tilde{\lambda} = \sqrt{\lambda^2 + k^2 q} - \lambda \sim \frac{k^2 q}{2 \lambda}.
\end{equation}

\n and the asymptotic equivalence holds as $\lambda \to \infty$.

\begin{remark} \label{REM alpha X}
\normalfont

In fact, there is second degree of freedom in $\CM$. It is not hard to see  that vectors in $\CM$ can be parametrized by

\begin{equation*}
\xi = \lambda e^{i \alpha} (-i, 1) + \tilde{\lambda} (\sin \alpha, \cos \alpha),
\end{equation*}

\n where $\lambda \in \RR$ and $\alpha \in [- \pi, \pi]$. However, we are not making use of the parameter $\alpha$ as it does not affect our first- and second-order nondegeneracy conditions. To be more precise, the $\alpha$-dependence fully factors out and cancels in these conditions. Therefore, we set $\alpha = 0$, which also simplifies the analysis. However, see the discussion preceding Remark~\ref{REM analyt of u^in}.
\end{remark}

Let $\Gamma$ be parametrized counterclockwise by $x(t) = (x_1(t), x_2(t))$ with $t \in [-\pi, \pi]$, and assume there are at most finitely many points where $x'=\frac{d}{dt}x$ vanishes. We may write the exponent in \eqref{test fctn xi} as

\begin{equation*}
i x(t) \cdot \xi = \lambda g(t) + i \tilde{\lambda} x_2(t),
\qquad \text{where} \qquad g(t) = x_1(t) + i x_2(t).
\end{equation*}

\n Next we consider the term multiplying the exponential inside the integral in \eqref{test fctn xi} and simplify it. To that end note that at any point where $x'$ doesn't vanish, the outer unit normal is given by $n = (x_2', -x_1') / |x'|$, so we may write

\begin{equation*}
-i \xi \cdot (x_2', -x_1') = i \lambda g'(t) + i \tilde{\lambda} x_1'(t).
\end{equation*}

\n Introducing further notation

\begin{equation} \label{v V}
v(t) = u^{\text{in}}(x(t)), \qquad V(t) = \nabla u^{\text{in}}(x(t)) 
\end{equation}

\n the integral in \eqref{test fctn xi} becomes

\begin{equation} \label{I}
I(\lambda)=\int_{-\pi}^\pi \left[ (x_2', -x_1') \cdot V +  i \lambda g' v + i \tilde{\lambda} x_1' v  \right] e^{\lambda g + i\tilde{\lambda} x_2} dt.
\end{equation}

\n In summary, if $u^{\text{in}}$ is nonscattering for the region $D$, then

\begin{equation*}
I(\lambda) = 0, \qquad \qquad \forall \ \lambda > 0.
\end{equation*}

\n Our goal is to study the asymptotic behavior of $I(\lambda)$ as $\lambda \to \infty$. Under suitable assumptions, we obtain the first- and second-order terms explicitly in the asymptotic expansion. If at least one of these terms is nonzero, we arrive at a contradiction and conclude that $u^{\text{in}}$ scatters.

The integral $I(\lambda)$ contains both real exponential and highly oscillatory terms as $g(t)$ is a complex-valued function. We analyze its asymptotic behavior using the method of steepest descent \cite{olver}, which requires the integrand to extend to the complex $t$-plane as an analytic function. The method of steepest descent is a powerful tool for studying the asymptotics of complex analytic integrals; however, it does not provide a general result applicable to an arbitrary region $D$. The asymptotic behavior depends on (a) the critical/saddle points (if any) of the analytic extension $g(t)$, i.e., the zeros of $g'(t)$, (b) the structure of the steepest descent paths, (c) the presence (or absence) of singularities, and (d) the contour deformations. Consequently, each example may present an interesting mathematical challenge.  

The outline of the paper is as follows. The main results and a representative set of applications are stated and carefully described in the following section (Section~\ref{mainsect}), however, most of the proofs are postponed to Section \ref{SECT proof circle} and Section \ref{SECT proofs 1st 2nd order}. In Section~\ref{SECT gen} we focus on the case where $g(t)$ has a simple saddle point (say at $t=t_0$) and formulate a nondegeneracy condition 

\begin{equation*}
(q-1) u^{\text{in}} (x(t_0)) x_2'(t_0) \neq 0
\end{equation*}

\n under which analytic incident waves always scatter (a second-order condition is also formulated). We apply this result to a true\footnote{i.e., the boundary $\Gamma$ is an ellipse that is not a circle.} elliptical region in Section~\ref{SECT ellipse} and to a nonconvex region (with regular boundary) in Section~\ref{SECT nonconvex}. In both of these cases, the point $x(t_0)$ lies in  $\CC^2 \backslash \RR^2$. In Section~\ref{SECT cusp}, we apply our result to regions with cusps -- both outward- and inward-pointing -- in which case $x(t_0) \in \RR^2$ corresponds precisely to the cusp point. The scattering result for outward-pointing cusps is known and follows from \cite{SS21}. However, our result for inward-pointing cusps is new. Our method of analysis can be generalized to obtain nondegeneracy conditions for regions with multiple and/or higher-order critical/saddle points. In Section~\ref{SEC corner}, we apply our approach to convex regions with a corner and establish (a slightly weaker version of) the result from \cite{BPS}, namely, that any incident wave that is nonvanishing at the corner is always scattered. We note that the asymptotic analysis in this case is simpler, as $\Re g$ is maximized at the endpoints of the two boundary integrals forming the corner. In particular, critical points and contour deformations are irrelevant in this case. The resulting nondegeneracy condition is analogous: $(q-1) u^{\text{in}}(p) \neq 0$, where $p$ represents the corner point.

We study the case of a disk in Section~\ref{SEC disk}. Even though a circle can be viewed as a degenerate limit of a true ellipse, the method of steepest descent used to analyze the true ellipse does not apply to the circle. For a true ellipse, the phase function $g(t)$ has a saddle point in the complex plane, which governs the leading behavior of the integral (after appropriate contour deformation). However, as the true ellipse approaches a circle, this saddle point disappears at infinity. We handle this degenerate case by manipulating the integral and changing the phase function, which creates a singularity for the integrand in the complex plane. We then analyze the residue of the integrand at this singularity. We study incident plane waves and establish that they always scatter, a result known in the literature. However, our method provides an alternative and elementary proof of this fact. We also study specific Herglotz waves that are known to be nonscattering for the disk at special values of $k$ and $q$. We rederive the equation that these parameter values must satisfy in order to produce nonscattering. This equation corresponds to the vanishing of the leading-order term in $I(\lambda)$, which, in this case, is equivalent to the vanishing of $I(\lambda)$ itself, as we derive a closed-form expression for it. Our methods can be extended to study broader classes of incident waves, which shall be addressed in the future.  
   
The above discussion highlights the delicate nature of the asymptotic behavior of $I(\lambda)$. Another point that further illustrates this, is the following: although $\tilde{\lambda} \to 0$ as $\lambda \to \infty$, the exponential term $e^{i \tilde{\lambda} x_2(t)}$ in the integral \eqref{I} cannot be dropped (replaced by 1). There are delicate cancellations happening in the integral which make the parameter $\tilde{\lambda}$ relevant and non-negligible.

We already mentioned the connection to the Schiffer and/or Pompeiu problem, which formally corresponds to setting $u^{\text{in}} = 1$ in the integral $I(\lambda)$ \cite{BST73, GS91, vog94}. Analyzing these problems via the asymptotic expansion of $I(\lambda)$ dates back to \cite{bern80}. In this context, the fact that $I(\lambda)$ cannot vanish identically for a true elliptical region follows from \cite{BST73}, relying on a closed-form expression of the integral in terms of the Bessel function $J_1$. Asymptotic methods based on the steepest descent technique were later employed in \cite{GS91} to reestablish this result as a particular case and to extend it far beyond elliptical regions (see also the subsequent works \cite{GS91TR, GS94}). The nonconvex region we study in Section~\ref{SECT nonconvex} is borrowed from \cite{GS91}. The disk is an exceptional region that does not possess the Pompeiu property (this observation goes back to \cite{Chak44}); that is, the integral $I(\lambda)$, with $u^{\text{in}} = 1$, can vanish identically if the radius of the disk is chosen appropriately. The Pompeiu/Schiffer conjecture states that the disk is the only domain (among bounded, simply connected Lipschitz domains) that does not possess the Pompeiu/Schiffer property. We refer the reader to Section \ref{appendix}, where we review the relation between the Pompeiu and Schiffer problems and give a brief survey of the main results.

In the scattering context, there is an additional degree of freedom -- namely, the incident wave $u^{\text{in}}$. As a result, the disk has a more nuanced behavior: it exhibits both positive and negative scattering outcomes depending on the choice of incident wave. For example, the disk always scatters plane waves, whereas certain Herglotz waves (for certain values of $k$ and $q$) are not scattered by the disk. As stated before, a natural question arises: ``among bounded simply connected Lipschitz domains with constant index of refraction $\ne 1$, is the disk the only domain that may fail to scatter?" Our work provides some evidence that, this question may be more likely than not to have an affirmative answer.

\section{Main results}
\label{mainsect}
\setcounter{equation}{0}

\subsection{The case of simple critical points} \label{SECT gen}

For this very useful case we make the following assumptions:

\begin{enumerate}

\item[(H1)] $\Gamma$ is a simple closed curve parametrized by $x(t) = (x_1(t), x_2(t))$ for $t \in [-\pi, \pi]$, where $x_j(t)$ for $j=1,2$ extend to complex analytic functions in an open set $\CT \subset \CC$, which contains the interval $[-\pi, \pi]$.  

\item[(H2)] $D$, the domain enclosed by $\Gamma$, has a constant refractive index $0<q \neq 1$. 

\item[(H3)] $g(t) = x_1(t) + ix_2(t)$ has a simple critical point at $t_0 \in \CT$, {\it i.e.}, $g'$ has a simple zero at $t_0$, and $t_0 \neq \pm \pi$.

\item[(H4)] There exists a contour $\PC$ in $\CT$ joining $-\pi$ to $\pi$ (and containing both these endpoints) that passes through $t_0$, such that

\begin{equation} \label{Re g < Re g t_0}
\Re g(t) < \Re g(t_0) \qquad \qquad \forall \ t \in \PC \backslash \{t_0\}.
\end{equation}

\item[(H5)] $u^{\text{in}}(x_1,x_2)$ extends to an analytic function of two complex variables in some open subset of $\CC^2$ containing

\begin{equation*}
\left\{(x_1(t), x_2(t)): t \in \CT \right\}.
\end{equation*}
 
\end{enumerate}

\begin{theorem} \label{THM gen}
Assume (H1)-(H5) and let $I(\lambda)$ be given by \eqref{I}. Then, as $\lambda \to \infty$

\begin{equation} \label{I C1 asymp}
I(\lambda) = \frac{1}{\lambda^\frac{3}{2}} e^{\lambda g(t_0)} \left[C_1 + O\left( \tfrac{1}{\lambda} \right) \right], 
~~ \hbox{ with } ~~
C_1 = k^2 (q-1) \sqrt{\frac{-2\pi}{g''(t_0)}}  u^{\text{in}}(x(t_0)) x_2'(t_0).
\end{equation}

\n In particular, under the nondegeneracy condition

\begin{equation} \label{ND 1}
(q-1) u^{\text{in}}(x(t_0)) x_2'(t_0) \neq 0
\end{equation}

\n the incident wave $u^{\text{in}}$ is scattered by $D$.

\end{theorem}

\begin{remark} \mbox{} \label{REM branch}
\normalfont

\begin{enumerate}
\item[$\bullet$] In forming $(-g'')^{\frac{1}{2}}$, the branch of $\omega_0 = \arg (-g''(t_0))$ must satisfy $|\omega_0 + 2\omega| \leq \pi/2$, where $\omega$ is the limiting value of $\arg (t-t_0)$ as $t \to t_0$ along the part of $\PC$ that joins $t_0$ to $\pi$. In other words, $\omega$ is the angle of slope of $\PC$ at $t_0$ from ``right" \cite{olver}. 

\item[$\bullet$] Note that the condition \eqref{ND 1} does not depend on the refractive index $q$, except for the requirement that $q\neq 1$, in which case the factor $q-1$ can be dropped.

\item[$\bullet$] Since $t_0$ is a critical point for $g(t)=x_1(t)+ix_2(t)$, it follows that $x_2'(t_0)=ix_1'(t_0)$, and thus 
the nondegeneracy condition (\ref{ND 1}) could equally well have been expressed using $x_1'(t_0)$ in place of $x_2'(t_0)$.
\end{enumerate}

\end{remark}

The above assumptions place us in a context where we can apply the steepest descent method. Namely, the analyticity in $\CT$ allows us to deform the original contour $[-\pi, \pi]$ into the contour $\PC$ while conserving the integral. The unique maximizer $t_0$ (of $\Re g$ along $\PC$), which is also a simple critical point of $g$, is an interior point of the contour $\PC$. If there are multiple maximizers $t_1,...,t_N$ (all simple critical  points and in the interior of $\PC$), the asymptotic formula readily generalizes as a sum of the contributions from the points $t_j$. The method of steepest descent can also be used to treat higher order critical points, however, we confine ourselves to simple critical points as this occurs more commonly in practice. Theorem~\ref{THM gen} is quite versatile. In Section~\ref{SECT ellipse} we apply it to the elliptical region, and in Section~\ref{SECT nonconvex} we apply it to a nonconvex regular region borrowed from \cite{GS91}. In these cases, the corresponding phase function $g$ has a simple nonreal saddle point $t_0$. Moreover, $x_2'(t_0) \neq 0$ and therefore, dropping this term from \eqref{ND 1}, we obtain that incident waves satisfying $u^{\text{in}}(x(t_0)) \neq 0$ always scatter from these regular regions.  

It may happen that $t_0$ is real, say, $t_0 \in (-\pi, \pi)$. In this case, the condition that $t_0$ is a saddle point of $g$ is equivalent to $x'(t_0) = 0$, meaning that the curve $\Gamma$ has a vanishing tangent at $t_0$ --  this occurs at a cusp. In Section~\ref{SECT cusp} we consider two model examples, the cardioid curve, which has an inward cusp, and the deltoid curve, which has an outward cusp. However, note that now $x_2'(t_0) = 0$ and as a result the nondegeneracy condition \eqref{ND 1} is not satisfied. Therefore, we need to consider the next term in the asymptotic expansion of $I(\lambda)$ and obtain a second-order nondegeneracy condition. This is done in Theorem~\ref{THM 2nd order} below, which we apply to regions with cusps. We remark that Theorem~\ref{THM 2nd order} is also useful for regular domains in the non-generic situations where $u^{\text{in}}(x(t_0)) = 0$.   

For the deltoid curve, it turns out that \eqref{Re g < Re g t_0} holds with $\PC = [-\pi,\pi]$ and therefore contour deformation is not needed. In this case (H5) is automatically satisfied (see Remark~\ref{REM analyt of u^in} below) and we obtain that any incident wave that is nonvanishing at one of the cusps of the deltoid always scatters. This fact is known and follows from \cite{SS21}. 

For the cardioid curve contour deformation is necessary to ensure \eqref{Re g < Re g t_0}. We show that any incident wave (satisfying an analyticity assumption) that is nonvanishing at the cusp of the cardioid always scatters. This result is new and does not follow from \cite{SS21}. In fact, we believe that this is a general feature of inward cusps. See Section~\ref{SECT card} for a more detailed discussion.

We remark that if the maximizer $t_0$ occurs at an ``endpoint" of the contour, {\it i.e.}, $t_0 = \pm \pi$, then the assumption that $t_0$ be a saddle point is not necessary (but the resulting asymptotic behavior is different). This situation arises in Section~\ref{SEC corner} when studying regions with a corner.

Finally, if we also introduce the parameter $\alpha$ as described in Remark~\ref{REM alpha X}, then the phase function $g$ must be replaced with $e^{i\alpha} g$. This, of course, does not affect the saddle point $t_0$. However, it does affect assumption (H4): deforming the contour $[-\pi,\pi]$ into $\PC$ is a topological problem that can be delicate. In particular, introducing $\alpha$ modifies the condition \eqref{Re g < Re g t_0}, which now reads: 

$$\Re \left[ e^{i\alpha} g(t) \right]  < \Re \left[ e^{i\alpha} g(t_0) \right] \qquad \qquad \forall \ t \in \PC \backslash \{t_0\}.$$ 

\n In general, the existence of such a contour $\PC$ may hold for some choices of $\alpha$ and fail for others. Thus, $\alpha$ affects both the geometry of the contour and the domain of analyticity $\CT$. For example, in the application to the elliptical region discussed in Section~\ref{SECT ellipse}, one can take any $\alpha \in \left(-\frac{\pi}{2}, \frac{\pi}{2} \right)$. Any such $\alpha \neq 0$ distorts the shapes of the curves in Figure~\ref{FIG el} and leads to an analogous result to Corollary~\ref{CORO ellipse}, but with a different (distorted) domain of analyticity $\CT$.

\begin{remark}[Analyticity of $u^{\text{in}}$] \label{REM analyt of u^in}
\normalfont \mbox{}

\n The incident wave $u^{\text{in}}$ is a solution of the Helmholtz equation everywhere, and hence is real analytic in $\RR^2$. In particular, $u^{\text{in}}$ admits an extension to an analytic function of two complex variables in an open set $\Omega \subset \CC^2$ containing $\RR^2$. This is sufficient when no contour deformation is required. The assumption (H5), ensures analyticity in a sufficiently large domain to allow for contour deformation. Note that, as a consequence of (H5), $u^{\text{in}}$ composed with $x(t)$, as well as all partial derivatives of $u^{\text{in}}$ composed with $x(t)$, are analytic functions of $t \in \CT$.  
\end{remark}

The condition \eqref{ND 1} is sufficient for scattering. It can happen that it is not satisfied, {\it i.e.}, $C_1=0$, which of course does not mean that $u^{\text{in}}$ is nonscattering. In this case, one can look at the next term in the asymptotic expansion of $I(\lambda)$. To that end, let us introduce

\begin{equation*}
f = \frac{k^2 q}{2} (x_2' + ix_1') v + V' \cdot (i,1).
\end{equation*}

\n In fact, the condition \eqref{ND 1} is equivalent to $f(t_0) \neq 0$ (cf. Lemma~\ref{LEM f(t0)}).

\begin{theorem} \label{THM 2nd order}
Assume (H1)-(H5) and let $I(\lambda)$ be given by \eqref{I}. Let $C_1$ be as in \eqref{I C1 asymp}. If $C_1 = 0$, then, as $\lambda \to \infty$

\begin{equation} \label{I next term asymp}
I(\lambda) = \frac{1}{\lambda^\frac{5}{2}} e^{\lambda g(t_0)} \left[C_2 + O\left( \tfrac{1}{\lambda} \right) \right],
\end{equation}

\n where

\begin{equation*}
C_2 = \frac{2\sqrt{\pi}}{\left(-2g''(t_0) \right)^{\frac{3}{2}}} 
\left[ f''(t_0) -f'(t_0) \frac{g'''(t_0)}{g''(t_0)} - ik^2 q g''(t_0) x_2'(t_0) V(t_0) \cdot (i,1) \right].
\end{equation*}

\n In particular, if $C_1 = 0$ and $C_2 \neq 0$, then the incident wave $u^{\text{in}}$ is scattered by $D$.

\end{theorem}

The power $(-g'')^{\frac{3}{2}}$ is taken using the same branch of the argument function as described in Remark~\ref{REM branch}. Unlike $C_1$, the expression for $C_2$ generally does not simplify, as it involves third-order derivatives of the boundary parametrization $x(t)$ and fourth-order derivatives of $u^{\text{in}}$. This theorem will be particularly useful for our regions with cusps, where the formula for $C_2$ does simplify.

Proofs of the above two theorems are given in Section~\ref{SECT proofs 1st 2nd order}. In the sections below we present applications of these results.

\begin{remark}[Examples of incident waves] \label{REM plane Herglotz waves}
\normalfont
\mbox{}

\n There are two important classes of incident waves:

\begin{enumerate}
\item[$\bullet$] Plane waves $u^{\text{in}}(x) = e^{ikx \cdot \eta}$, where $\eta \in \RR^2$ is a unit vector; or a linear combination of plane waves with different direction vectors $\eta$.

\item[$\bullet$] Herglotz waves -- ``continuous" superposition of plane waves, {\it i.e.},

\begin{equation} \label{Herglotz}
u^{\text{in}}(x) = \int_{\mathbb{S}^1} \psi(\eta) e^{ikx \cdot \eta} ds(\eta),
\end{equation}

\n where $\mathbb{S}^1$ denotes the unit circle centered at the origin and $\psi \in L^2(\mathbb{S}^1)$ is called the density function. Note that a plane wave formally corresponds to a Herglotz wave with a Dirac delta as its density function.

\end{enumerate}

\n For such incident waves the assumption (H5) is automatically satisfied, as $u^{\text{in}}(x)$ extends to an entire function of $x \in \CC^2$. 
\end{remark}

\begin{remark} (Plane waves always scatter)
\normalfont

\n Plane waves never vanish, therefore they satisfy \eqref{ND 1}, provided $x_2'(t_0) \neq 0$. In particular, plane waves always scatter from an elliptical region and the nonconvex region of Section~\ref{SECT nonconvex}. Moreover, plane waves satisfy the second-order condition $C_2 \neq 0$ for the cusped regions considered in Section~\ref{SECT cusp}, and thus always scatter from those as well. Finally, they also always scatter from regions with corners considered in Section~\ref{SEC corner}. 
\end{remark}

\subsection{Application: true ellipse} \label{SECT ellipse} 

Let $D$ be the region bounded by the true ellipse

\begin{equation} \label{Gamma ellipse}
x(t) = \left(a \cos t, b \sin t \right), \qquad \qquad t \in [-\pi, \pi].
\end{equation}

\n where, without loss of generality we may assume $a>b>0$. In particular the boundary $\Gamma = \partial D$ is not a circle. The saddle point equation $g'(t) = x_1'(t) + i x_2'(t)= 0$ reads

\begin{equation*}
\tan t = i \frac{b}{a} \qquad \Longleftrightarrow \qquad t = t_0 + \pi n, \quad n \in \ZZ,
\end{equation*}

\n where

\begin{equation*}
t_0 = i \tanh^{-1} \frac{b}{a}.
\end{equation*}

\n Note that for the circle $b=a$, and there are no solutions to the above equation as formally  $t_0 = i \infty$. Straightforward calculations show that

\begin{equation*}
x(t_0) = \frac{1}{\sqrt{a^2 - b^2}} \left(a^2, i b^2 \right)
\qquad \text{and} \qquad x'(t_0) = \frac{ab}{\sqrt{a^2 - b^2}} (-i,1).
\end{equation*}

\n Clearly $t_0$ is a simple critical point as 

\begin{equation*}
g(t_0) = \sqrt{a^2 - b^2} \qquad \text{and} \qquad g''(t_0) = -g(t_0) \neq 0.
\end{equation*}

\n With $t = r + i s$,

\begin{equation*}
\Re g(t) = \cos r \left( a \cosh s - b \sinh s \right). 
\end{equation*}

\begin{figure}[H]
    \centering
    \includegraphics[width=0.49\linewidth]{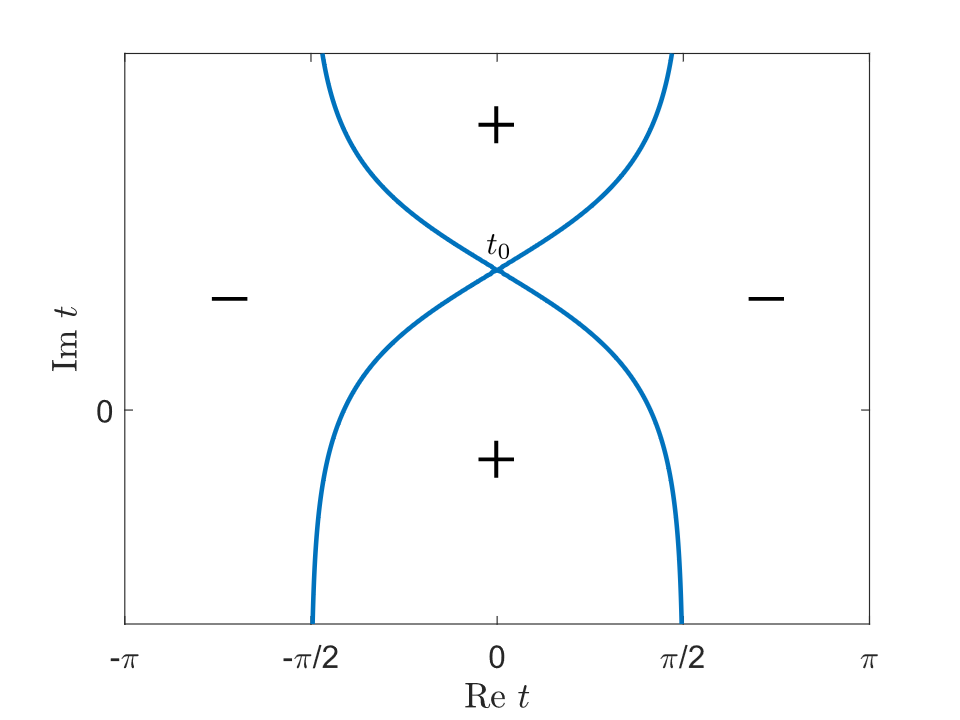}
    \includegraphics[width=0.495\linewidth]{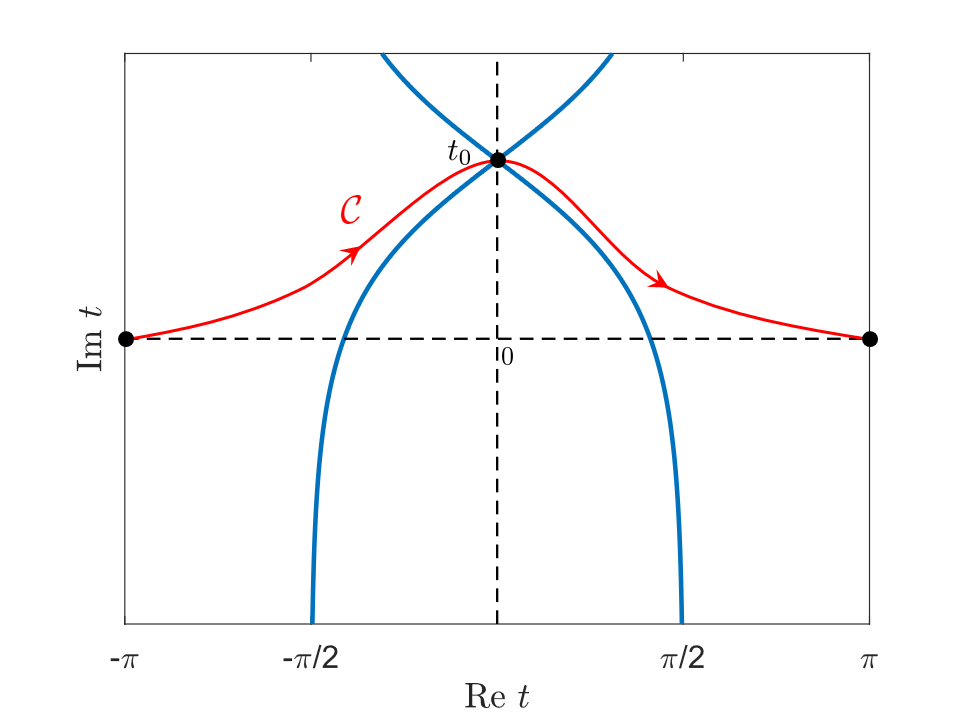}
    \caption{\textbf{Left:} The curve $\Re g(t) = \Re g(t_0)$, with ``$-$" (resp. ``$+$") marks indicating the region where $\Re g(t) < \Re g(t_0)$ (resp. $\Re g(t) > \Re g(t_0)$). \textbf{Right:} The contour $\PC$ lying in the ``$-$" region except at the point $t_0$.}
    \label{FIG el}
\end{figure}

\n The key is to find a contour $\PC$ that satisfies assumption (H4) of Section~\ref{SECT gen}. To that end, in Figure~\ref{FIG el}, we plot the curves $\Re g(t) = \Re g(t_0)$ and mark the region where $\Re g(t) < \Re g(t_0)$ with a ``$-$". The goal is to connect $-\pi$ to $\pi$ while remaining entirely inside the ``$-$" region, except at the crossing point $t_0$, which is a saddle point for $\Re g$. As shown in the right-hand image of Figure~\ref{FIG el}, this is clearly possible.

Regarding assumption (H5), let us describe the smallest region of analyticity that makes Theorem~\ref{THM gen} applicable. As we always have analyticity in an open set around $\RR$, we just need to concentrate on the region $\CT_0$ defined to be the closure of the domain lying above the $\Re t$ axis and below the curve $\Re g(t) = \Re g(t_0)$. Explicitly,

\begin{equation} \label{T0 ellipse}
\CT_0 = \left\{ r + is \ : \ |r| \leq \arccos \sqrt{1 - \frac{b^2}{a^2}}, \qquad 0\leq s \leq \ln \left( \frac{1 - |\sin r|}{\cos r} \sqrt{\frac{a+b}{a-b}} \right) \right\}.
\end{equation}

\n Finally, as $x_2'(t_0) \neq 0$, this factor can be dropped from the nondegeneracy condition \eqref{ND 1}. Theorem~\ref{THM gen} immediately implies:

\begin{corollary} \label{CORO ellipse}
Let $D$ be the region enclosed by the ellipse \eqref{Gamma ellipse} (with $a>b>0$) having a constant refractive index $0<q \neq 1$. Let $u^{\text{in}}$ be an incident wave that extends to an analytic function of two complex variables in an open set of $\CC^2$ containing $\left\{(x_1(t), x_2(t)): t \in \CT_0 \right\}$. Assume that   

\begin{equation} \label{ND ellipse}
u^{\text{in}} \left( \frac{a^2}{\sqrt{a^2 - b^2}}, \frac{i b^2}{\sqrt{a^2 - b^2}} \right) \neq 0,
\end{equation}

\n then $u^{\text{in}}$ is scattered by $D$. 

\end{corollary}

Consider the Herglotz wave \eqref{Herglotz}. If we parametrize $\eta = (\cos \alpha, \sin \alpha)$, then the density function $\psi$ can be viewed as a function of $\alpha$. In other words,

\begin{equation} \label{Herglotz with tau}
u^{\text{in}}(x) = \int_{-\pi}^{\pi} \psi(\alpha) e^{ikx \cdot (\cos \alpha, \sin \alpha)} d\alpha.
\end{equation}

\n Let us now expand $\psi$ into its Fourier series

\begin{equation*}
\psi(\alpha) = \sum_{n \in \ZZ} \psi_n e^{in\alpha},
\end{equation*}

\n where $\psi_n \in \CC$ denote the Fourier coefficients. Substituting the latter into \eqref{Herglotz with tau} gives

\begin{equation} \label{u^in with H_n}
u^{\text{in}}(x) = \sum_{n \in \ZZ} \psi_n h_n(x),
\end{equation}

\n where $h_n$ denotes the Herglotz wave with density $e^{in\alpha}$, {\it i.e.},

\begin{equation} \label{H_n}
h_n(x) = \int_{-\pi}^{\pi} e^{in\alpha} e^{ikx \cdot (\cos \alpha, \sin \alpha)} d\alpha = 2\pi i^n e^{in \theta} J_n(kr),
\end{equation}

\n where $r, \theta$ denote the polar coordinates of $x$, $J_n$ denotes the Bessel function of order $n$, and the last equation is a direct consequence of a well-known integral representation of $J_n$. The formula \eqref{u^in with H_n} shows the importance of the Herglotz wave functions $h_n$. These functions are known to be nonscattering for the disk for particular choices of $k$ (depending on $q, n$ and the radius of the disk) \cite{CK} (see also Section~\ref{SEC disk}). Let us now study the scattering of these incident waves from the true ellipse $D$. We calculate

\begin{equation*}
h_n \left( \frac{a^2}{\sqrt{a^2 - b^2}}, \frac{i b^2}{\sqrt{a^2 - b^2}} \right) = 2\pi i^n\exp\left\{ -n  \tanh^{-1} \frac{b^2}{a^2} \right\} J_n\left( k \sqrt{a^2+b^2} \right).
\end{equation*}

\n Consequently, the nondegeneracy condition \eqref{ND ellipse} is equivalent to

\begin{equation} \label{J_n ellipse}
J_n \left( k \sqrt{a^2 + b^2} \right) \neq 0.
\end{equation}

\n Clearly, for a fixed ellipse there are infinitely many $k>0$ for which the above condition fails. Suppose $k>0$ is such that $J_n \left( k \sqrt{a^2 + b^2} \right) = 0$. We can consider the second-order nondegenarcy condition of Theorem~\ref{THM 2nd order}. Calculations (performed with the maple-scripts found at \cite{MAPL}) show that

\begin{equation*}
C_2 = \frac{3 i^n (2\pi)^\frac{3}{2} (q-1)k^3 ab}{2(a^2+b^2)^\frac{3}{2} (a^2-b^2)^\frac{5}{4}} \left[a^4 - \frac{2}{3}(n-2) a^2b^2 + b^4 \right] e^{-n  \tanh^{-1} \frac{b^2}{a^2}} J_{n+1}\left( k \sqrt{a^2+b^2} \right). 
\end{equation*}

\n Using the fact that $J_n$ and $J_{n+1}$ do not have common zeros, the condition $C_2 = 0$ is equivalent to the expression in square brackets being equal to zero. Equivalently, this holds if $\mu=a^2/b^2$ solves the quadratic equation

\begin{equation} \label{mu quadratic}
\mu^2 - \frac{2}{3}(n-2) \mu + 1 = 0.
\end{equation}

\n It is easy to see that for $n < 5$, the above quadratic equation does not have real solutions and for $n=5$ the unique solution is $\mu=1$ which corresponds to the disk. Consequently, $C_2\neq 0$ and $h_n$ always scatters in the case $n \leq 5$ (no matter what the parameters $q, a$ and $b$ are). If, on the other hand, $n > 5$, then \eqref{mu quadratic} has two positive solutions: one less than 1 and the other greater. Since we assumed that $a>b$, let us consider the solution greater than one:

\begin{equation} \label{mu_n}
\mu_n = \frac{n-2 + \sqrt{(n-5)(n+1)}}{3}.
\end{equation}

\n Then $h_n$ scatters, provided that the ratio $a^2/b^2$ is different from $\mu_n$. We summarize this discussion in the following.

\begin{corollary}
Let $D$ be the region enclosed by the ellipse \eqref{Gamma ellipse} (with $a>b>0$)  having a constant refractive index $0<q \neq 1$. Let $n \in \ZZ$ and $u^{\text{in}} = h_n$ be given by \eqref{H_n}. Then $u^{\text{in}}$ is scattered by $D$ provided one of the following holds: 

\begin{itemize}
\item[$\bullet$] $J_n \left( k \sqrt{a^2 + b^2} \right) \neq 0$, or
\item[$\bullet$] $n \leq 5$, or
\item[$\bullet$] $n > 5$ and $a^2/b^2 \neq \mu_n$, where $\mu_n$ is defined by \eqref{mu_n}.
\end{itemize}

\end{corollary}

We remark that if $n > 5$, $a^2/b^2 = \mu_n$ and $k$ is such that $J_n \left( k \sqrt{a^2 + b^2} \right) = 0$, then $C_1 = C_2 = 0$, {\it  i.e.}, both nondegeneracy conditions fail. To analyze the scattering behavior of $h_n$ in this nongeneric situation, further terms in the asymptotic expansion of $I(\lambda)$ are needed. This is a task for the future.

\subsection{Application: regular nonconvex region} \label{SECT nonconvex}

The example here is borrowed from \cite{GS91}. Consider the nonconvex region $D$ (see Figure~\ref{FIG ncv}) bounded by the curve 

\begin{equation} \label{ncv}
x(t) = (2 + \cos 2t) \left( \cos t, \sin t \right), \qquad \qquad t \in [-\pi, \pi].
\end{equation}

\begin{figure}[H]
    \centering
    \includegraphics[width=0.49\linewidth]{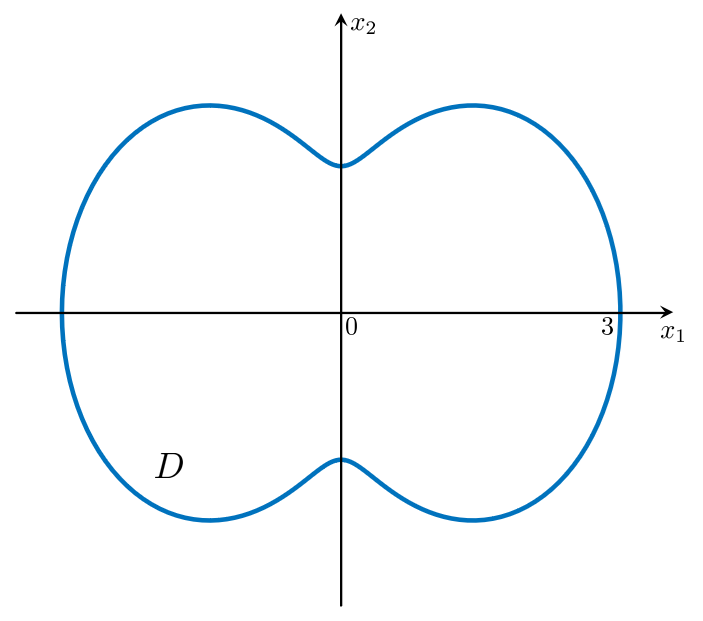}
    \caption{Region $D$ bounded by \eqref{ncv}.}
    \label{FIG ncv}
\end{figure}

\n Now, $g(t) = (2 + \cos 2t) e^{it}$ and direct calculation shows that

\begin{equation} \label{t0 ncv}
t_0 = \frac{i}{2} \ln (2+\sqrt{7})
\end{equation}

\n is a simple critical point and that $x_2'(t_0) \neq 0$. Figure~\ref{FIG ncv1} shows a contour deformation satisfying (H4). As a consequence we obtain the following result, which  for simpliciy we state only for the incident waves described in Remark~\ref{REM plane Herglotz waves},

\begin{corollary}
Let $D$ be the region enclosed by the curve \eqref{ncv}, having a constant refractive index $0<q \neq 1$. Let $u^{\text{in}}$ be a finite linear combination of plane waves or a Herglotz wave such that

\begin{equation*}
u^{\text{in}}(x(t_0)) \neq 0,
\end{equation*}

\n where $t_0$ is given by \eqref{t0 ncv}, then $u^{\text{in}}$ is scattered by $D$.
\end{corollary}

\begin{figure}[H]
    \centering
    \includegraphics[width=0.49\linewidth]{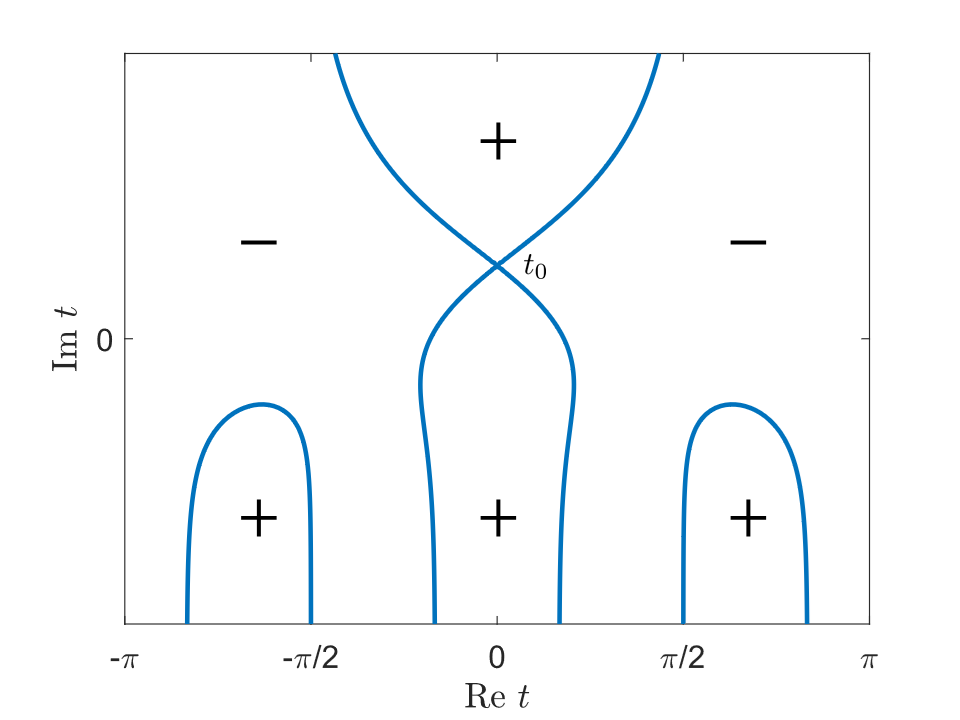}
    \includegraphics[width=0.495\linewidth]{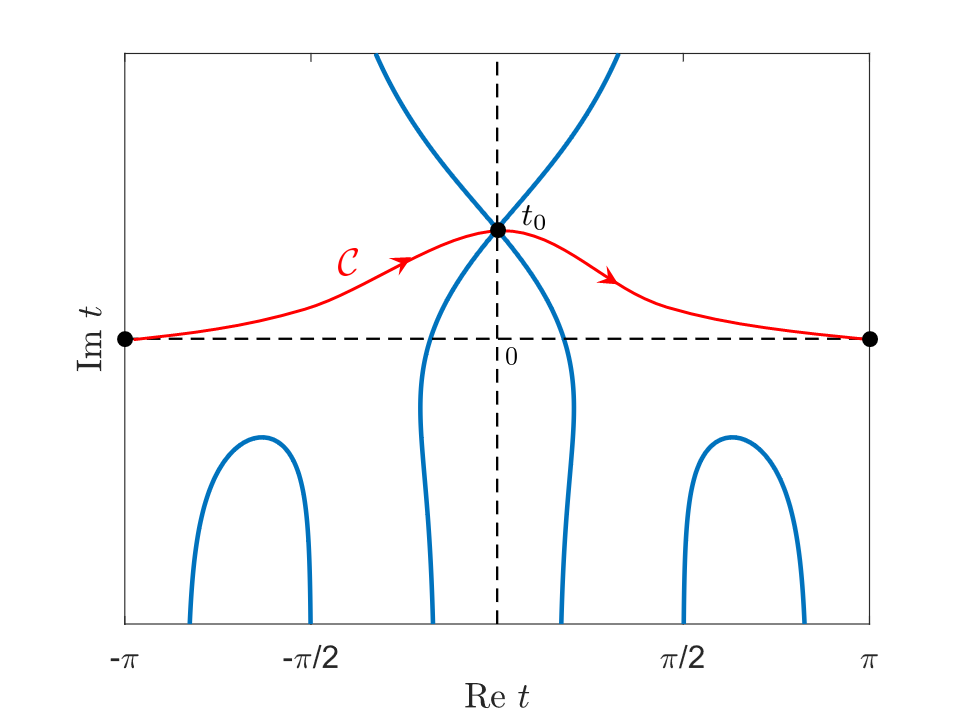}
    \caption{\textbf{Left:} The curve $\Re g(t) = \Re g(t_0)$, with ``$-$" (resp. ``$+$") marks indicating the region where $\Re g(t) < \Re g(t_0)$ (resp. $\Re g(t) > \Re g(t_0)$). \textbf{Right:} The contour $\PC$ lying in the ``$-$" region except at the point $t_0$.}
    \label{FIG ncv1}
\end{figure}

\subsection{Application: curves with cusps} \label{SECT cusp}

Before presenting our two examples, let us make a connection with two results from \cite{SS21, SS24}. In Theorem 1.4 of \cite{SS21} the authors prove that if a given region (k-quadrature domain) $D$ admits a nonscattering incident wave that is nonvanishing at a point $x_0 \in \partial D$ then there are two possibilities: either 

\begin{enumerate}
\item[(a)] $D$ is regular near $x_0$, or 

\item[(b)] The complement of $D$ is thin near $x_0$, {\it e.g.}, $D$ has an inward cusp at $x_0$ (for precise definition see \cite{SS21}). 
\end{enumerate}

\n In particular, a region with an outward cusp does not fall in  either of the above categories and therefore will always scatter incident waves that are nonvanishing at the cusp. In Section~\ref{SECT delt} we consider a region bounded by a deltoid curve -- which has an outward cusp -- and reestablish the always-scattering result (under a nonvanishing condition) for this specific case. 

More interestingly, in Section~\ref{SECT card} we consider a region bounded by a cardioid curve, which has an inward cusp. Such a region falls into category (b) above, and therefore nonscattering from it is not excluded by the result of \cite{SS21}. We establish a new result that rules out nonscattering from this particular region with an inward cusp. Specifically, we show that under a certain analyticity assumption on the incident wave -- which holds for physically relevant waves, such as linear combinations of plane waves or Herglotz waves (cf. Remark~\ref{REM plane Herglotz waves}) -- the region bounded by a cardioid always scatters such a wave, provided it is nonvanishing at the cusp. We believe that this is a general feature of inward cusps and that, similar to outward cusps, they almost always scatter. More specifically, we believe that case (b) in the above result of \cite{SS21} can be excluded (at least for physically relevant incident waves and a constant refractive index). 

In relation to our result, we mention Corollary 1.9 of \cite{SS24} (see also Remark 2.4 in \cite{SS21}), where the authors consider a region $D$ with an inward cusp (a $k$-quadrature domain to be precise) and assume the existence of an incident wave satisfying

\begin{equation} \label{u^in pos}
u^{\text{in}} > 0 \qquad \text{on} \ \partial D.
\end{equation}

\n Then they show that there exists a refractive index function $q \in L^\infty(D)$ such that $u^{\text{in}}$ does not scatter from $D$. The existence of an incident wave satisfying \eqref{u^in pos} is a nontrivial question. Such a wave exists provided $D$ is contained in a disk of radius less than $k^{-1} j_{0,1}$, where $j_{0,1}$ is the first positive zero of the Bessel function $J_0$ \cite{SS21}, however, the existence of such an incident wave is unclear for general cusped domains. For noncircular domains, physical waves may generically fail to satisfy \eqref{u^in pos}.

\subsubsection{Inward cusp} \label{SECT card}

Let $D$ be the region bounded by the cardioid (see Figure~\ref{FIG card}) 

\begin{equation} \label{card}
x(t) = (1 - \cos t) \left(\cos t, \sin t \right), \qquad \qquad t \in [-\pi, \pi].
\end{equation}

\n We may rewrite

\begin{equation*}
g(t)=x_1(t) + ix_2(t) = -\frac{1}{2} \left( e^{it} - 1 \right)^2.
\end{equation*}

\n In particular $t_0 = 0$, which corresponds to the cusp, is a simple saddle point, as $g'(0)=0$ and $g''(0) =1 \neq 0$. Although $t_0$ is an interior point of the interval $[-\pi,\pi]$, $\Re g$ is not maximized there (see Figure~\ref{FIG card}) and therefore we will need to deform this interval to a contour $\PC$ on which $t_0$ is a maximizer of $\Re g$, {\it i.e.}, assumption (H4) is satisfied. With $t=r+is$, direct simplification gives

\begin{figure}[H]
    \centering
    \includegraphics[width=0.49\linewidth]{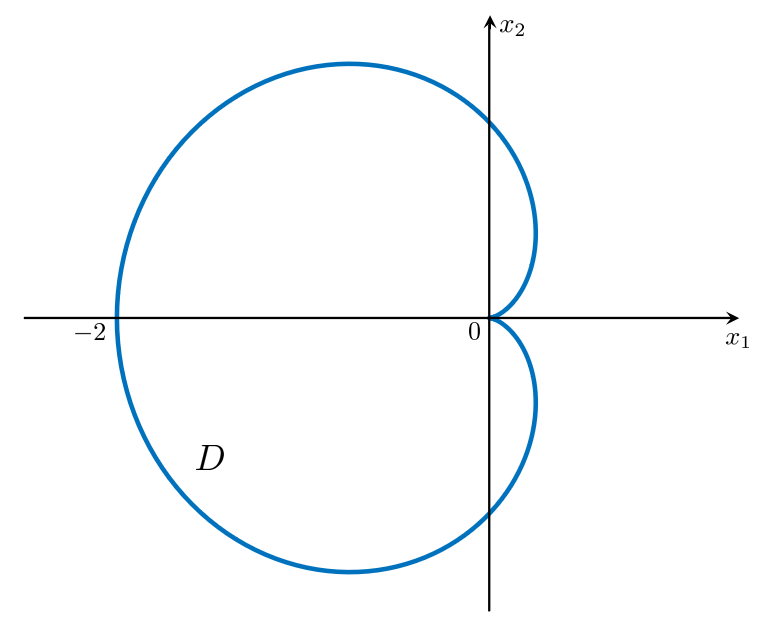}
    \includegraphics[width=0.495\linewidth]{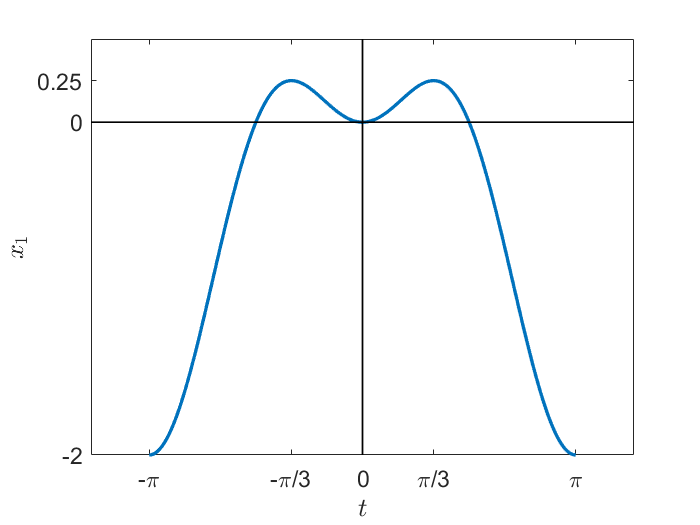}
    \caption{\textbf{Left:} Region $D$ bounded by the cardioid. \textbf{Right:} The plot of $\Re g$.}
    \label{FIG card}
\end{figure}

\begin{equation*}
\Re g(t) = e^{-s} \cos r  - \frac{1}{2} e^{-2s} \cos 2r  - \frac{1}{2}.
\end{equation*}

\n Note that $\Re g(0) = 0$ and our goal is to find a contour $\PC$ that joins $-\pi$ to $\pi$ and passes through $0$ such that $\Re g(t) < 0$ everywhere on $\PC \backslash \{0\}$. To that end, in Figure~\ref{FIG card 2} we plot the curve $\Re g(t) = 0$, which can be equivalently written as 

\begin{equation*}
s = \ln \left( \cos r \pm |\sin r| \right).
\end{equation*}

\n We mark the region where $\Re g(t) < 0$ with a ``$-$". The right-hand image of Figure~\ref{FIG card 2} shows the desired contour $\PC$. 

\begin{figure}[H]
    \centering
    \includegraphics[width=0.49\linewidth]{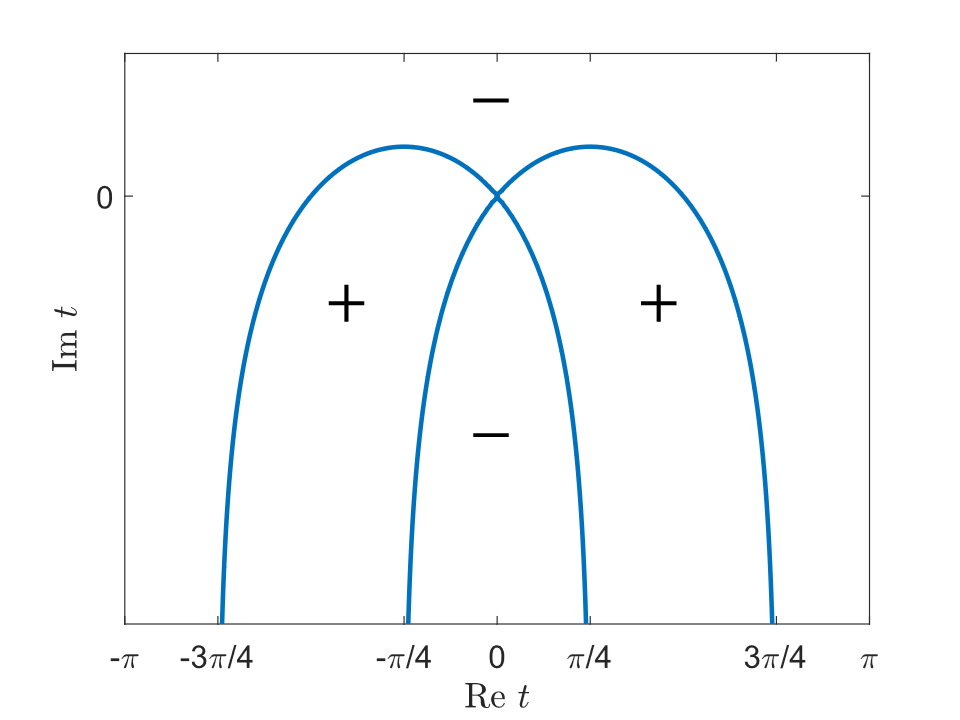}
    \includegraphics[width=0.495\linewidth]{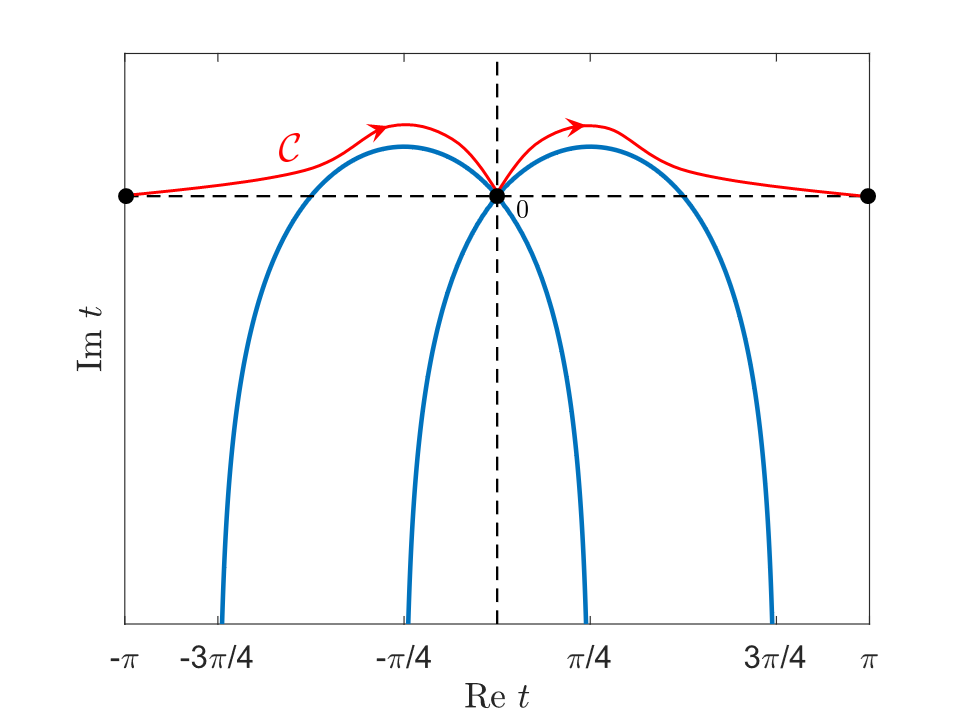}
    \caption{\textbf{Left:} The curve $\Re g(t) = 0$, with ``$-$" (resp. ``$+$") marks indicating the region where $\Re g(t) < 0$ (resp. $\Re g(t) > 0$). \textbf{Right:} The contour $\PC$ lying in the ``$-$" region except at the point $0$.}
    \label{FIG card 2}
\end{figure}

\n To make the contour deformation from $[-\pi,\pi]$ to $\PC$ possible, it remains to impose analyticity in the closure of the domain lying above the $\Re t$ axis and below the curve $\Re g(t) = 0$. Explicitly,

\begin{equation} \label{T0 card}
\CT_0 = \left\{ r + is \ : \ |r| \leq \frac{\pi}{2}, \qquad 0\leq s \leq \ln \left( \cos r + |\sin r| \right) \right\}.
\end{equation}

\n Note that $t_0 = 0$ corresponds to the cusp $x(0) = (0,0)$ and $x_2'(0) = 0$. Consequently, as already discussed, the nondegeneracy condition \eqref{ND 1} fails and Theorem~\ref{THM gen} does not apply. Therefore, we apply Theorem~\ref{THM 2nd order}. Calculations (performed with the maple-scripts found at \cite{MAPL}) give that

\begin{equation*}
C_2 = 3i \sqrt{\frac{\pi}{2}} k^2 (q-1) u^{\text{in}}(0).
\end{equation*}

\n Now Theorem~\ref{THM 2nd order} readily implies:

\begin{corollary} \label{CORO card}
Let $D$ be the region enclosed by the cardioid \eqref{card}, having a constant refractive index $0<q \neq 1$. Let $u^{\text{in}}$ be an incident wave that extends to an analytic function of two complex variables in an open set of $\CC^2$ containing $\left\{(x_1(t), x_2(t)): t \in \CT_0 \right\}$. Assume that   

\begin{equation} \label{ND card}
u^{\text{in}} (0) \neq 0,
\end{equation}

\n then $u^{\text{in}}$ is scattered by $D$. 

\end{corollary}

\subsubsection{Outward cusp} \label{SECT delt}

Let $D$ be the region bounded by the deltoid (see Figure~\ref{FIG delt}) 

\begin{equation} \label{delt}
x(t) = \left( 2 \cos t + \cos 2t, 2 \sin t - \sin 2t \right), \qquad \qquad t \in [-\pi, \pi].
\end{equation}

\begin{figure}[H]
    \centering
    \includegraphics[width=0.49\linewidth]{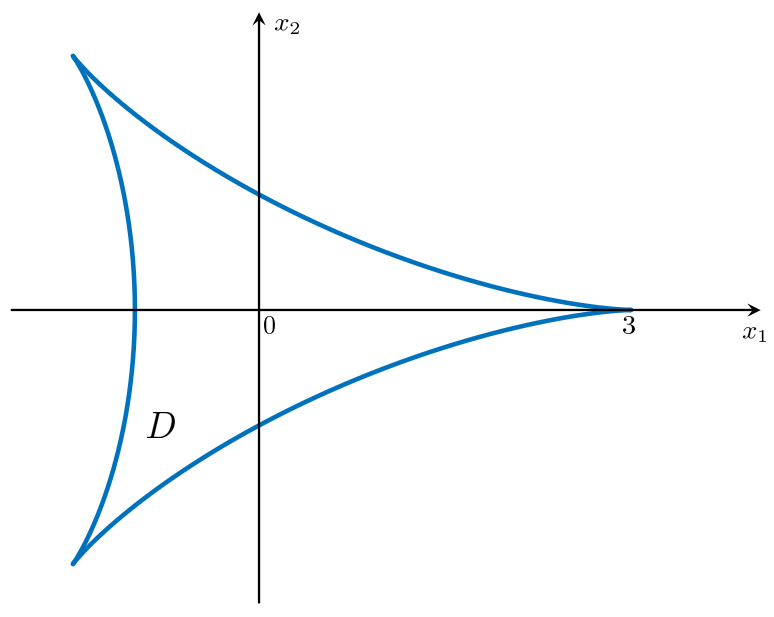}
    \includegraphics[width=0.495\linewidth]{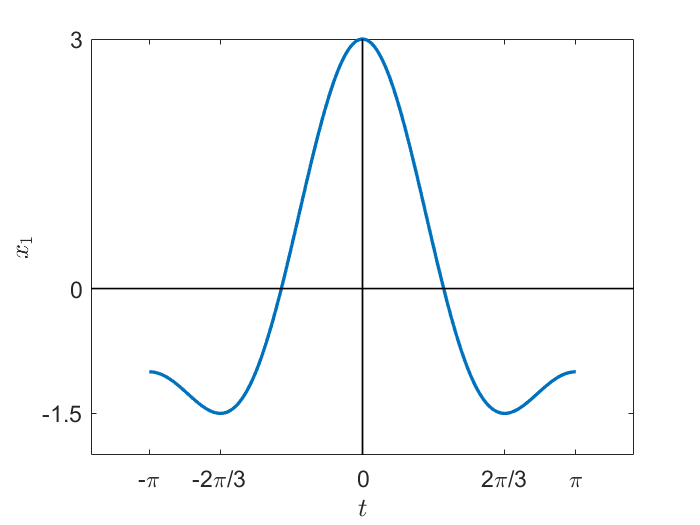}
    \caption{\textbf{Left:} Region $D$ bounded by the deltoid. \textbf{Right:} The plot of $\Re g$.}
    \label{FIG delt}
\end{figure}

\n Note that $\Re g (t) = x_1(t)$, $t \in [-\pi,\pi]$, has a unique maximizer at $t_0=0$ (as can be seen from Figure~\ref{FIG delt}), which also happens to be a simple critical point. Indeed,

\begin{equation*}
g'(t) = 4 \left[i (1-\cos t) - \sin t \right] \left(\frac{1}{2} + \cos t \right)
\end{equation*}

\n and so $g'(0) = 0$, with $g''(0) = -6 \neq 0$. As was already discussed, no contour deformation is necessary in this case, and we can simply take $\PC = [-\pi, \pi]$. In addition, the analyticity assumption (H5) is automatically satisfied. Note that $t_0$ corresponds to the cusp $x(t_0) = (3,0)$ and $x_2'(t_0) = 0$. Consequently, as in the previous case Theorem~\ref{THM gen} does not apply, and we use Theorem~\ref{THM 2nd order}. Calculations (performed with the maple-scripts found at \cite{MAPL}) show that

\begin{equation*}
C_2 = \sqrt{\frac{\pi}{12}} k^2 (q-1) u^{\text{in}}(3,0).
\end{equation*}

\n Therefore, we obtain the following:

\begin{corollary}
Let $D$ be the region enclosed by the deltoid \eqref{delt}, having a constant refractive index $0<q \neq 1$. Let $u^{\text{in}}$ be an incident wave such that

\begin{equation*}
u^{\text{in}}(3,0) \neq 0,
\end{equation*}

\n then $u^{\text{in}}$ is scattered by $D$.

\end{corollary}

\subsection{Application: convex region with a straight-sided corner} \label{SEC corner}

Our applications so far have been to domains whose boundaries are given by an analytic parametrization $(x_1(t),x_2(t))$. This was motivated by the fact that, for Lipschitz domains it is already known that, singularities generically scatter.   In this section, we take a step back and illustrate, how the technique developed here can also be used to prove that Lipschitz point singularities generically scatter. For this purpose, we consider a convex region with a smooth boundary except for a straight corner. This scattering result was already established in \cite{BPS} (in $\RR^d$ for a smooth refractive index function $q$). Our example considers the simplified setting: $\RR^2$ with a constant refractive index.

\begin{corollary} \label{CORO corner}
Let $D$ be a convex region with a smooth boundary except for a straight corner at the point $p$, and suppose it has a constant refractive index $0<q \neq 1$. Let $u^{\text{in}}$ be an incident wave such that

\begin{equation*}
u^{\text{in}}(p) \neq 0,
\end{equation*}

\n then $u^{\text{in}}$ is scattered by $D$.

\end{corollary}

Theorem~\ref{THM gen} is not directly applicable in this setting; however, the same ideas still apply. We again analyze the asymptotic behavior of $I(\lambda)$ defined in \eqref{I}. To illustrate the approach, we first place the origin of the coordinate system at the corner point $p$, and select the two orthogonal axis so that the setup  matches that  shown in Figure~\ref{FIG corner}. That is, $D$ lies to the left of the origin. The straight line segments $L_1$ and $L_2$ form the corner, whose angular size is $2\theta$, and the $x_1$-axis bisects the corner into two equal angles of size $\theta \in \left(0, \frac{\pi}{2} \right)$. We can split $I$ into the sum of three integrals: 

\begin{equation*}
I(\lambda) = I_1(\lambda) + I_2(\lambda) + R(\lambda),
\end{equation*}

\n where $I_j$ integrates over the line segment $L_j$ for $j=1,2$, and $R$ integrates over the rest of the boundary $L = \partial D \backslash \left( L_1 \cup L_2 \right)$. We choose a parametrization so that $t=0$ corresponds to the corner. Note that in this case $\Re g(t) = x_1(t)$ is maximized at the unique point $t=0$, and $\Re g(0) = 0$. On the other hand, for all points on $L$

\begin{equation*}
\Re g(t) < \max\{a_1, a_2\} < 0.
\end{equation*}

\begin{figure}[H]
    \centering
    \includegraphics[width=0.7\linewidth]{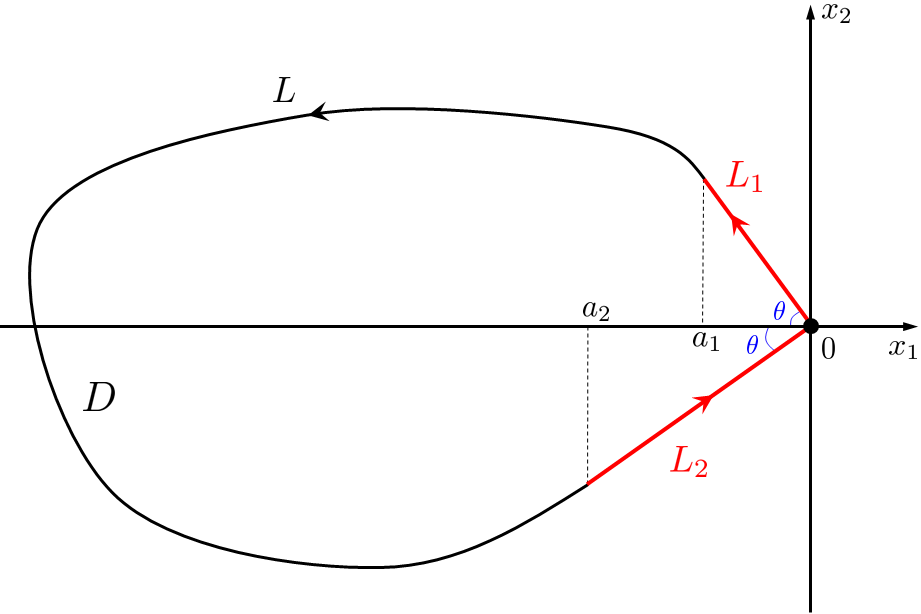}
    \caption{The convex region $D$ with a corner at $0$. Line segments $L_1$ and $L_2$ form the corner, which is bisected by the $x_1$-axis into two equal angles of size $\theta \in \left(0, \frac{\pi}{2} \right)$.}
    \label{FIG corner}
\end{figure}

\n In other words, $R$ is an exponentially small term. The leading behavior comes from the sum $I_1 + I_2$. For each $I_j$, the phase function $g(t)$ is analytic and $\Re g$ is maximized at the endpoint $t=0$. Therefore, the method of steepest descent — or more precisely, Laplace’s method for contour integrals \cite{olver} — applies, and we obtain the asymptotic expansion of $I_j$ for $j=1,2$. It turns out that the first two leading-order terms of $I_1$ and $I_2$ are opposites and cancel in the sum $I_1+I_2$. By contrast, the third-order terms in $I_1$ and $I_2$ add up and do not cancel. This is precisely due to the presence of a corner at $t=0$.

\begin{theorem} \label{THM corner}
Let $D$ be as in Figure~\ref{FIG corner} with constant index of refraction $0<q \neq 1$, and let $I(\lambda)$ be the corresponding integral given by \eqref{I}. Then, as $\lambda \to \infty$

\begin{equation*}
I(\lambda) = \frac{C}{\lambda^2} +  O\left( \lambda^{-3} \right), 
\end{equation*}

\n where

\begin{equation*}
C = 2k^2 \frac{\tan \theta}{1+ \tan^2 \theta} (q-1) u^{\text{in}}(0).
\end{equation*}

\end{theorem}

\vspace{.1in}

\n Corollary~\ref{CORO corner} follows immediately from the above theorem, whose proof is given in Section~\ref{SECT corner proof}. We observe that in the limit as $\theta$ approaches $\frac{\pi}{2}$ — that is, when the corner disappears — the constant $C$ tends to 0. Finally, if $u^{\text{in}}(p) = 0$ in Corollary~\ref{CORO corner}, then one may consider the next term in the asymptotic expansion of $I(\lambda)$ to obtain a second-order nondegeneracy condition, as was done in earlier applications.

\subsection{Application: circle} \label{SEC disk}

Let $D$ be the unit disk centered at the origin. In this case $x(t) = (\cos t, \sin t)$ and $g(t) = e^{it}$ has no critical points in $\CC$. Therefore, Theorem~\ref{THM gen} is not directly applicable. We introduce the notation
\begin{equation} \label{I cal}
\PI(\lambda, u^{\text{in}}) = \int_D u^{\text{in}}(x) e^{ix\cdot \xi} dx ,
\end{equation}
where the vector $\xi$ is given by \eqref{xi with lambda} and  \eqref{lambda tild asymp}. In Section~\ref{SECT Prelim} we reduced this integral to a boundary integral. Specifically, 

\begin{equation*}
\PI(\lambda, u^{\text{in}}) = \frac{1}{(q-1)k^2} I(\lambda, u^{\text{in}}),
\end{equation*}

\n where $I$ is defined by \eqref{I}. The nonscattering of $u^{\text{in}}$ implies that the integral $\PI$, or equivalently, $I$ vanishes for all $\lambda>0$. Consider a single incident plane wave with incidence angle $\alpha \in [-\pi, \pi]$:

\begin{equation*}
u^{\text{in}}(x) = e^{ikx \cdot \eta}, \qquad \qquad \eta = (\cos \alpha, \sin \alpha).
\end{equation*}

\n For simplicity of notation set

\begin{equation*}
\PI(\lambda, \alpha) = \PI(\lambda, e^{ikx \cdot \eta})=\int_D e^{ix \cdot (k \cos \alpha-i\lambda\,,\,k \sin \alpha + \lambda + \tilde{\lambda})} dx.
\end{equation*}

\n For plane waves it is more convenient to work with $\PI$ rather than $I$, because it can be reduced to a boundary integral in a slightly different way, that is better suited for the analysis.

\begin{lemma} \label{LEM circle plain wave}
Let $J_1$ be the Bessel function of the first kind of order 1, then

\begin{equation} \label{I circle LEM}
\PI(\lambda, \alpha) = \frac{\pi}{\sqrt{c_{\lambda, \alpha}}} J_1 \left( 2 \sqrt{c_{\lambda, \alpha}} \right),
\end{equation}

\n where, with $\tilde{\lambda}$ given by \eqref{lambda tild asymp},

\begin{equation*}
c_{\lambda, \alpha} = - \frac{ik}{2} e^{i \alpha} \lambda + \frac{\lambda \tilde{\lambda}}{2} + \frac{k^2}{4} + \frac{\tilde{\lambda} k}{2} \sin \alpha + \frac{\tilde{\lambda}^2}{4}.
\end{equation*}

\n In particular, as $\lambda \to \infty$, uniformly in $\alpha$,

\begin{equation*}
c_{\lambda, \alpha} \sim - \frac{ik}{2} e^{i \alpha} \lambda.   
\end{equation*}

\end{lemma}

\begin{remark} \label{REM circle} \mbox{}
\normalfont

\begin{enumerate}

\item[$\bullet$]  The function $\frac{J_1(\sqrt{z})}{\sqrt{z}}$ (or more generally, $\frac{J_n(\sqrt{z})}{z^{n/2}}$ for $n=0,1,...$) is entire. Therefore, the formula \eqref{I circle LEM} is unaffected by the choice of the branch of the square root function.

\item[$\bullet$] This lemma follows directly from \eqref{J_1 disk} in the Appendix. However, we also provide an alternative proof as a special case of a more general formula, which we will use to analyze the incident Herglotz waves described below.

\end{enumerate}

\end{remark}

This lemma readily shows that $\PI$ cannot vanish for all $\lambda>0$, and consequently plane waves always scatter from $D$. This result is known in the literature and follows from the radial symmetry and the unique determination result of $q$ from the scattering data for the disk $D$ \cite{Buk}. Lemma~\ref{LEM circle plain wave} provides an alternative and elementary proof of this fact, which, as we saw in previous sections, generalizes to non radially symmetric regions.  

Regarding more general incident waves, note that the integral \eqref{I cal} is linear with respect to $u^{\text{in}}$. As a consequence, a linear combination of plane waves with direction angles $\alpha_j$ corresponds to a linear combination of the integrals $\PI(\lambda, \alpha_j)$. For a Herglotz wave \eqref{Herglotz} with a density function $\psi$, nonscattering implies that

\begin{equation} \label{Herglotz nonscat}
\int_{-\pi}^{\pi} \psi(\alpha) \PI(\lambda, \alpha) d \alpha = 0, \qquad \qquad \forall \ \lambda>0,
\end{equation}
where as before we write $\psi(\alpha)$ instead of $\psi(\cos \alpha,\sin \alpha)$. The asymptotic behavior of the above integral can then be studied as $\lambda \to \infty$. In view of Lemma~\ref{LEM circle plain wave} we can use the large argument approximation of the Bessel function $J_1$ and substitute it in \eqref{Herglotz nonscat}, in place of $\PI$. This will be addressed elsewhere in the future. Here we confine ourselves to analyzing the special Herglotz waves $h_n$ \eqref{H_n}: for $n \in \ZZ$

\begin{equation} \label{Bessel inc}
u^{\text{in}}(x) = h_n(x) = 2\pi i^n  e^{in \theta} J_n(kr),
\end{equation}

\n where $r, \theta$ denote the polar coordinates of $x$ and $J_n$ is the Bessel function of order $n$. As already discussed, \eqref{Bessel inc} is a Herglotz wave with density function

\begin{equation} \label{psi Bessel}
\psi(\alpha) = e^{in \alpha}.
\end{equation}

\n For this incident wave, instead of working with $\PI$ ({\it i.e.}, analyzing the integral \eqref{Herglotz nonscat} with density function \eqref{psi Bessel}) it is more convenient to work with $I$.

\begin{lemma} \label{LEM Herglotz disk}
Let $I(\lambda)$ be the integral \eqref{I} for the incident wave \eqref{Bessel inc}, where $n \in \ZZ$. Let $\tilde{\lambda}$ be given by \eqref{lambda tild asymp}, then

\begin{equation*}
I(\lambda) = 4\pi^2 C k \left(-\frac{i \tilde{\lambda}}{k \sqrt{q}} \right)^n
\end{equation*}

\n where

\begin{equation*}
C =  J_n'(k) J_n(k \sqrt{q}) - \sqrt{q} J_n(k) J_n'(k \sqrt{q}).
\end{equation*}

\end{lemma}

\n In particular, if $C \neq 0$ the incident wave \eqref{Bessel inc} scatters. However, it may happen that $C=0$. In fact, for any fixed $n \geq 0$ and $q \neq 1$ there are infinitely many positive $k_j \to \infty$ for which $C=0$, and at these special wavenumbers the incident wave \eqref{Bessel inc} is nonscattering \cite{CK}. This result was established using the radial symmetry of the disk and separation of variables. Here we see how the Wronskian $C$ not unexpectedly arises in the integral $I(\lambda)$.

The above two lemmas are proven in Section~\ref{SECT proof circle}.

\section{Proofs of Lemmas~\ref{LEM Herglotz disk} and \ref{LEM circle plain wave}} \label{SECT proof circle}
\setcounter{equation}{0}

\subsubsection{Proof of Lemma~\ref{LEM Herglotz disk}}

Let us simplify the integral $I(\lambda)$ \eqref{I} for the incident wave \eqref{Bessel inc} and the circle $x(t) = (\cos t, \sin t)$. First, note that

\begin{equation*}
v(t) = u^{\text{in}}(x(t)) = 2\pi i^n J_n(k) e^{int}.
\end{equation*}

\n Now

\begin{equation*}
\begin{split}
u^{\text{in}}_{x_1}(x(t)) = 2\pi i^n \left[ k J_n'(k) \cos t - in J_n(k) \sin t \right] e^{int}
\\[.1in]
u^{\text{in}}_{x_2}(x(t)) = 2\pi i^n \left[ k J_n'(k) \sin t + in J_n(k) \cos t \right] e^{int}
\end{split}
\end{equation*}

\n and consequently

\begin{equation*}
(x_2', -x_1') \cdot V = (\cos t, \sin t) \cdot \nabla u^{\text{in}}(x(t)) = 2\pi i^n k J_n'(k) e^{int}.      
\end{equation*}

\n Thus, we obtain

\begin{equation*}
I(\lambda) = 2\pi i^n \int_{-\pi}^\pi e^{int} \left[ k J_n'(k) - \lambda J_n(k) e^{it} - i \tilde{\lambda} J_n(k) \sin t \right] \exp\left\{\lambda e^{it} + i \tilde{\lambda} \sin t\right\} dt.
\end{equation*}

\n Let us now change the integration variable to $z = e^{it}$. As $t$ is real, $\overline{z} = z^{-1}$, and we may use the identities

\begin{equation} \label{sin cos z}
\cos t = \frac{1}{2} \left( z + \frac{1}{z} \right)
\qquad \text{and} \qquad
\sin t = \frac{1}{2i} \left( z - \frac{1}{z} \right),
\end{equation}

\n to rewrite

\begin{equation*}
I(\lambda) = 2\pi i^{n-1} \int_{\PC_1} z^{n-1} \left[ k J_n'(k) - \lambda J_n(k) z - \frac{\tilde{\lambda}}{2} J_n(k) \left( z - \frac{1}{z} \right) \right] \exp\left\{ \left( \lambda + \tfrac{\tilde{\lambda}}{2} \right) z - \frac{\tilde{\lambda}}{2z}\right\} dz,
\end{equation*}

\n where $\PC_1$ denotes the unit circle centered at the origin traversed counterclockwise. We can conveniently write the above formula as

\begin{equation} \label{I using A disk}
\frac{1}{4\pi^2 i^n} I(\lambda) = k J_n'(k) A_n(\lambda) - \left( \lambda + \tfrac{\tilde{\lambda}}{2} \right) J_n(k) A_{n+1}(\lambda) + \tfrac{\tilde{\lambda}}{2} J_n(k) A_{n-1}(\lambda),
\end{equation}

\n where

\begin{equation*}
A_n(\lambda) =  \frac{1}{2\pi i}\int_{\PC_1} z^{n-1} \exp\left\{ \left( \lambda + \tfrac{\tilde{\lambda}}{2} \right) z - \frac{\tilde{\lambda}}{2z} \right\} dz .   
\end{equation*}

\n The goal now is to understand the behavior of this integral as $\lambda \to \infty$. The following lemma shows an interesting connection to the Bessel functions:

\begin{lemma} \label{LEM Bessel integral}
For  $n \in \ZZ$, let $J_n$ denote the Bessel function of order $n$.
Assume $a, b \in \CC$, then

\begin{equation} \label{Bessel int}
\frac{1}{2\pi i}\int_{\PC_1} z^{n-1} \exp\left\{a z - \frac{b}{z}\right\} dz  = \frac{J_{|n|} \left(2\sqrt{ab}\right)}{\left(ab\right)^{|n|/2}} \cdot
\begin{cases}
\displaystyle (-b)^{n} , \qquad &n \geq 0
\\[.2in]
\displaystyle a^{-n}, &n \leq 0
\end{cases}
\end{equation}

\end{lemma}

\begin{remark} 
\normalfont The formula \eqref{Bessel int} is unaffected by the choice of the branch of the square root function (cf. Remark~\ref{REM circle}). If we choose the principal branch of the square root function, then for any $a,b \in \CC \backslash \RR_-$ with $ab \in \CC \backslash \RR_-$, the right-hand side of \eqref{Bessel int} simplifies to

\begin{equation*}
(-1)^n J_n(2\sqrt{ab}) \left( \frac{b}{a} \right)^{\frac{n}{2}},
\end{equation*}

\n which holds for any $n \in \ZZ$.

\end{remark}

\begin{proof}
Let us present the proof for $n \geq 0$. The case $n < 0$ follows by a change of variables and the fact that $J_{-n}=(-1)^nJ_n$. We start with the expansion

\begin{equation*}
z^{n-1} \exp\left\{a z - \frac{b}{z}\right\} = z^{n-1} \sum_{m,l=0}^\infty \frac{a^m (-b)^l}{m! l!} \frac{z^m}{z^l}
\end{equation*}

\n By the residue theorem, in order to evaluate the integral over $\PC_1$,  we just need to find the coefficient of $1/z$ in the above expansion. This is obtained by letting $l = m + n$:

\begin{equation} \label{ab series}
\sum_{m=0}^\infty \frac{a^m (-b)^{m+n}}{m! (m+n)!}.   
\end{equation}

\n To express this sum in terms of the Bessel function, we first recall that

\begin{equation*}
J_n(x) = \sum_{m=0}^\infty \frac{(-1)^{m} }{m! (m+n)!} \left( \frac{x}{2} \right)^{2m+n} = \left( \frac{x}{2} \right)^{n} \sum_{m=0}^\infty \frac{(-1)^{m} }{m! (m+n)!} \left( \frac{x^2}{4} \right)^{m}.
\end{equation*}

\n But then

\begin{equation*}
\sum_{m=0}^\infty \frac{(-1)^{m}}{m! (m+n)!} \zeta^m = \frac{J_n(2\sqrt{\zeta})}{\zeta^{n/2}}.
\end{equation*}

\n To conclude the proof we apply this to \eqref{ab series} with $\zeta = ab$.
\end{proof}





Direct simplification gives that

\begin{equation*}
\left( \lambda + \frac{\tilde{\lambda}}{2} \right) \frac{\tilde{\lambda}}{2} = \frac{k^2 q}{4}.
\end{equation*}

\n Set $c = k \sqrt{q}/2$. Using Lemma~\ref{LEM Bessel integral} and \eqref{I using A disk} we then get

\begin{equation*}
\begin{split}
\frac{I(\lambda)}{4\pi^2 \left( - \frac{i\tilde{\lambda}}{2} \right)^n} &= k J_n'(k) \frac{J_n(2c)}{c^n} + c^2 J_n(k) \frac{J_{n+1}(2c)}{c^{n+1}} - J_n(k) \frac{J_{n-1}(2c)}{c^{n-1}} =
\\[.1in]
&= k J_n'(k) \frac{J_n(2c)}{c^n} + \frac{J_n(k)}{c^{n-1}} \left[ J_{n+1}(2c) - J_{n-1}(2c) \right] =
\\[.1in]
& = \frac{1}{c^n} \left[ k J_n'(k) J_n(2 c) - 2c J_n(k) J_n'(2 c) \right],
\end{split}
\end{equation*}

\n where in the last step we used the recurrence relation $J_{n+1} - J_{n-1} = -2 J_n'$.

\subsection{Proof of Lemma~\ref{LEM circle plain wave}}

\n We work with the integral

\begin{equation} \label{PI disk}
\PI(\lambda, \alpha)=\int_D \exp \left\{ ix \cdot \left( k \cos \alpha - i\lambda, k \sin \alpha + \lambda + \tilde{\lambda} \right) \right\} dx.
\end{equation}

\n Note that for any $\xi = (\xi_1, \xi_2) \in \CC^2$ we can write

\begin{equation*}
e^{i x \cdot \xi} = \frac{1}{\xi_1 + i \xi_2} \div \left[(-i,1) e^{i x \cdot \xi}\right] .
\end{equation*}

\n Consequently, the divergence theorem implies

\begin{equation*}
\int_D e^{i x \cdot \xi} dx = -\frac{1}{\xi_1 + i \xi_2} \int_\Gamma e^{i x \cdot \xi} \left( dx_1 + idx_2 \right).
\end{equation*}

\n Using this identity in \eqref{PI disk} we arrive at

\begin{equation} \label{I with B}
\PI(\lambda, \alpha) = -\frac{1}{k e^{i\alpha} + i \tilde{\lambda}} B(\lambda, \alpha),
\end{equation}

\n where

\begin{equation*}
B(\lambda, \alpha) = \int_{\Gamma} \exp \left\{ \lambda (x_1 + ix_2) + ik (x_1 \eta_1 + x_2 \eta_2) + i\tilde{\lambda} x_2 \right\} \left( dx_1 + idx_2 \right).
\end{equation*}

\n As $\Gamma$ is the unit circle parametrized by $x_1 = \cos t$ and $x_2 = \sin t$ with $t \in [-\pi, \pi]$, the above integral becomes

\begin{equation} \label{B circle}
B(\lambda, \alpha) = i \int_{-\pi}^\pi \exp \left\{ \lambda e^{it} + ik (\eta_1 \cos t + \eta_2 \sin t) + i\tilde{\lambda} \sin t \right\} e^{it} dt
\end{equation}

\n Setting $z = e^{it}$ and using \eqref{sin cos z} we get

\begin{equation*}
\eta_1 \cos t + \eta_2 \sin t = \frac{\eta_1 - i \eta_2}{2} z +\frac{\eta_1 + i \eta_2}{2z} = \frac{e^{-i\alpha}}{2} z + \frac{e^{i\alpha}}{2z}.
\end{equation*}

\n Thus, \eqref{B circle} can be simplified to

\begin{equation*} 
B(\lambda, \alpha) = \int_{\PC_1} \exp \left\{a z - \frac{b}{z} \right\} dz,
\end{equation*}

\n where $\PC_1$ denotes the unit circle centered at the origin traversed counterclockwise and where, suppressing the dependence on $\lambda$ and $\alpha$,

\begin{equation*}
a = \lambda + \frac{\tilde{\lambda}}{2} + \frac{ik}{2} e^{-i\alpha} ,
\qquad \qquad
b = \frac{\tilde{\lambda}}{2} - \frac{ik}{2} e^{i\alpha}.
\end{equation*}

\n Lemma~\ref{LEM Bessel integral} now gives

\begin{equation*}
B(\lambda, \alpha) = -2\pi i b \frac{J_1(2\sqrt{a b})}{\sqrt{a b}}.
\end{equation*}

\n To conclude the proof of Lemma~\ref{LEM circle plain wave} we substitute this identity into \eqref{I with B} and simplify the product $a b$.

\section{Proofs of Theorems~\ref{THM gen} and \ref{THM 2nd order}} \label{SECT proofs 1st 2nd order}
\setcounter{equation}{0}

\subsection{Proof of Theorem~\ref{THM gen}}

Assume the setting of Theorem~\ref{THM gen}. As previously discussed, the imposed assumptions ensure that the interval $[-\pi,\pi]$ in the defining integral of $I(\lambda)$ in \eqref{I} can be deformed to the contour $\PC$ while conserving the integral. The method of steepest descent then allows us to capture the asymptotic behavior by localizing the integral around the critical point $t_0$. In spite of the factor  $\lambda$, the second term within the square brackets in \eqref{I} does not in itself determine the leading-order behavior. To see this we integrate by parts 

\begin{equation*}
\int_{-\pi}^\pi \lambda g' v e^{\lambda g + i \tilde{\lambda} x_2} dt = \int_{-\pi}^\pi \left(e^{\lambda g} \right)' v e^{i \tilde{\lambda} x_2} dt = - \int_{-\pi}^\pi \left( v' + i \tilde{\lambda} v x_2' \right) e^{\lambda g + i\tilde{\lambda} x_2} dt.
\end{equation*}

\n We now insert the above identity in \eqref{I} and use the fact that $v' = V \cdot x'$ to obtain

\begin{equation*}
I(\lambda) =
\int_{-\pi}^\pi \left[ \left( (x_2', -x_1') - i (x_1', x_2') \right) \cdot V + \tilde{\lambda} v \left( x_2' + i x_1' \right) \right] e^{\lambda g + i \tilde{\lambda} x_2} dt.
\end{equation*}

\n Note that

\begin{equation*}
(x_2', -x_1') - i (x_1', x_2') =(-i, -1) (x_1' + i x_2')  = (-i, -1) g',
\end{equation*}

\n and therefore

\begin{equation*}
I(\lambda) = 
\frac{1}{\lambda} \int_{-\pi}^\pi \left[ \left(e^{\lambda g} \right)' e^{i \tilde{\lambda} x_2} V \cdot (-i,-1) + \lambda \tilde{\lambda}v (x_2' + ix_1') e^{\lambda g + i \tilde{\lambda} x_2}\right] dt.  
\end{equation*}

\n Now integration by parts and multiplication by $\lambda$ gives

\begin{equation} \label{I 0.5}
\lambda I(\lambda) = \int_{-\pi}^\pi \left[ \lambda \tilde{\lambda} (x_2' + ix_1') v + V' \cdot (i,1) + i \tilde{\lambda} x_2' V \cdot (i,1)  \right] e^{\lambda g + i \tilde{\lambda} x_2} dt.
\end{equation}

\n From the asymptotic relation \eqref{lambda tild asymp} we know that $\lambda \tilde{\lambda}$ behaves like a constant, therefore the leading contribution should come from the first two terms in the square brackets in \eqref{I 0.5}. Specifically, as $\lambda \to \infty$, 

\begin{equation*}
\lambda \tilde{\lambda} \sim \frac{k^2 q}{2}. 
\end{equation*}

\n Therefore, we add and subtract this constant from $\lambda \tilde{\lambda}$  to rewrite \eqref{I 0.5} as

\begin{equation} \label{I 1}
\lambda I(\lambda) = \int_{-\pi}^\pi f  e^{\lambda g + i\tilde{\lambda} x_2} dt +  \int_{-\pi}^\pi \left[ \left(\lambda \tilde{\lambda} - \tfrac{k^2 q}{2} \right) (x_2' + ix_1') v + i\tilde{\lambda} x_2' V \cdot (i,1)  \right] e^{\lambda g + i\tilde{\lambda} x_2} dt, 
\end{equation}

\n with

\begin{equation*}
f = \frac{k^2 q}{2} (x_2' + ix_1') v + V' \cdot (i,1).
\end{equation*}

\n Finally, we introduce

\begin{equation} \label{h}
h_\lambda(t) = e^{i \tilde{\lambda} x_2(t)} - 1.
\end{equation}

\n Deforming the contour of integration, we now arrive at the following decomposition

\begin{equation} \label{I decomp}
\lambda I(\lambda) = I_0(\lambda) + R_1(\lambda) + R_2(\lambda),
\end{equation}

\n where

\begin{equation} \label{I0 R1}
I_0(\lambda) = \int_{\PC} f  e^{\lambda g} dt,
\qquad \qquad
R_1(\lambda) = \int_{\PC} f  e^{\lambda g} h_\lambda dt
\end{equation}

\n and

\begin{equation} \label{R2}
R_2(\lambda) = \int_{\PC} \left[ \left(\lambda \tilde{\lambda} - \tfrac{k^2 q}{2} \right) (x_2' + ix_1') v + i\tilde{\lambda} x_2' V \cdot (i,1)  \right] e^{\lambda g + i\tilde{\lambda} x_2} dt.
\end{equation}

\n If $f(t_0) \neq 0$, the method of steepest descent implies that, as $\lambda \to \infty$ (cf. Theorem 7.1 in \cite{olver})

\begin{equation*}
I_0(\lambda) \sim \sqrt{\frac{-2\pi}{\lambda g''(t_0)}} e^{\lambda g(t_0)} f(t_0),
\end{equation*}

\n where an appropriate branch of the square root must be chosen (see Remark~\ref{REM branch}). Let us show that the remaining terms in \eqref{I decomp} are much smaller than $\frac{1}{\sqrt{\lambda}} e^{\lambda g(t_0)}$. Regarding $R_1$, we first note that as $\PC$ is a finite contour, $h_\lambda$ can be made small as $\lambda \to \infty$ uniformly in $t$. Namely, there exists $\lambda_0, M >0$ such that for all $\lambda > \lambda_0$

\begin{equation} \label{h bound}
|h_\lambda(t)| \leq \frac{M}{\lambda}, \qquad \qquad \forall \ t \in \PC.
\end{equation} 

\n Hence,

\begin{equation*}
|R_1(\lambda)| \leq \frac{M|\PC|}{\lambda} e^{\lambda \Re g(t_0)} \max_{\PC} |f|  \ll \frac{1}{\sqrt{\lambda}} e^{\lambda \Re g(t_0)}. 
\end{equation*}

\n Regarding $R_2$, we have

\begin{equation} \label{lambda lambda_tild}
\lambda \tilde{\lambda} - \frac{k^2 q}{2} \sim - \frac{k^4 q^2}{8 \lambda^2} 
\end{equation}

\n and therefore an analogous estimate gives that $|R_2(\lambda)|   \ll \frac{1}{\sqrt{\lambda}} e^{\lambda \Re g(t_0)}$. In combination we obtain

\begin{equation*}
\lambda I(\lambda) \sim \sqrt{\frac{-2\pi}{\lambda g''(t_0)}} e^{\lambda g(t_0)} f(t_0).     
\end{equation*}

\n It remains to rewrite $f(t_0)$ in a more convenient form in terms of the incident wave $u^{\text{in}}$.

\begin{lemma}\label{LEM f(t0)}
One has the formula $\displaystyle f(t_0) = k^2 (q-1) u^{\text{in}}(x(t_0)) x_2'(t_0).$
\end{lemma}

\begin{proof}
Since $t_0$ is a critical point $g'(t_0) = 0$, {\it i.e.}, $x_1'(t_0) + i x_2'(t_0) = 0$. By definition 

\begin{equation*}
V(t) = (V_1(t), V_2(t)) = \left( \partial_{x_1}u^{\text{in}}(x(t)), \partial_{x_2}u^{\text{in}}(x(t))  \right).
\end{equation*}

\n Therefore,

\begin{equation*}
V'(t) \cdot (i,1) = iV_1'(t) + V_2'(t) = \left[ i \nabla \partial_{x_1}u^{\text{in}}(x(t)) + \nabla \partial_{x_2}u^{\text{in}}(x(t))  \right] \cdot x'(t).
\end{equation*}

\n Now evaluating the above expression at $t_0$, replacing $x_1'(t_0) = -i x_2'(t_0)$ and simplifying the dot product, we see that the mixed partial derivatives cancel, and we obtain

\begin{equation*}
V'(t_0) \cdot (i,1) = \Delta u^{\text{in}}(x(t_0))  x_2'(t_0) = - k^2 u^{\text{in}}(x(t_0))  x_2'(t_0), 
\end{equation*}

\n where in the last step we used the equation \eqref{u^in Helm} for $u^{\text{in}}$. Recalling that $v(t_0) = u^{\text{in}}(x(t_0))$ we then arrive at

\begin{equation*}
f(t_0) = k^2 q u^{\text{in}}(x(t_0)) x_2'(t_0)  - k^2 u^{\text{in}}(x(t_0))  x_2'(t_0), 
\end{equation*}

\n which concludes the proof.

\end{proof}

\subsection{Proof of Theorem~\ref{THM 2nd order}} \label{SECT next term}

Our starting point is the decomposition \eqref{I decomp}, and we assume that $f(t_0) = 0$. In this case, as $\lambda \to \infty$ (see \cite{olver})

\begin{equation*}
I_0(\lambda) = \frac{1}{\lambda^{\frac{3}{2}}} e^{\lambda g(t_0)} \left[ c_0 + O\left( \tfrac{1}{\lambda} \right)\right],
\end{equation*}

\n where

\begin{equation} \label{c0}
c_0 = \frac{2\sqrt{\pi}}{\left(-2g''(t_0) \right)^{\frac{3}{2}}} \left[ f''(t_0) - f'(t_0) \frac{g'''(t_0)}{g''(t_0)} \right].
\end{equation}

\n Now $I_0$ alone does not determine the leading-order behavior in \eqref{I decomp}. The second term in the integral \eqref{R2} defining $R_2(\lambda)$  also contributes to the leading behavior. To see this, we write

\begin{equation*}
R_2(\lambda) = R_{21}(\lambda) + R_{22}(\lambda), 
\end{equation*}

\n where

\begin{equation*}
R_{21}(\lambda) = i \frac{k^2 q}{2\lambda} \int_{\PC} x_2' V \cdot (i,1) e^{\lambda g} dt
\end{equation*}

\n and 

\begin{equation*}
\begin{split}
R_{22}(\lambda) =& \int_{\PC} \left[ \left(\lambda \tilde{\lambda} - \frac{k^2 q}{2} \right) (x_2' + ix_1') v + i \left(\tilde{\lambda} - \frac{k^2 q}{2 \lambda} \right) x_2' V \cdot (i,1)  \right] e^{\lambda g + i\tilde{\lambda} x_2} dt +
\\[.1in]
&+ i \frac{k^2 q}{2\lambda} \int_{\PC} x_2' V \cdot (i,1) e^{\lambda g} h_\lambda dt,
\end{split}
\end{equation*}

\n with $h_\lambda$ defined by \eqref{h}. Applying the method of stationary phase to $R_{21}$ we conclude that, as $\lambda \to \infty$

\begin{equation*}
R_{21}(\lambda) = \frac{1}{\lambda^{\frac{3}{2}}} e^{\lambda g(t_0)} \left[ c_1 + O\left( \tfrac{1}{\lambda} \right)\right]
\end{equation*}

\n with

\begin{equation} \label{c1}
c_1 = i\sqrt{\pi}  k^2 q  \frac{x_2'(t_0) V(t_0) \cdot (i,1)}{\left(-2g''(t_0) \right)^{\frac{1}{2}}}.
\end{equation}

\n Consequently, we can rewrite \eqref{I decomp} as follows

\begin{equation} \label{I decomp 2}
\lambda I(\lambda) = \frac{1}{\lambda^{\frac{3}{2}}} e^{\lambda g(t_0)} \left[ c_0 + c_1 + O\left( \tfrac{1}{\lambda} \right)\right] + R_1(\lambda) + R_{22}(\lambda).
\end{equation}

\n Assume now

\begin{equation*}
C_2 = c_0 + c_1 \neq 0
\end{equation*}

\n and let us show that the first term determines the leading behavior of the right-hand side of \eqref{I decomp 2}. Because of the estimates \eqref{h bound}, \eqref{lambda lambda_tild}, the nature of the contour $\PC$, and the asymptotic relation 

\begin{equation*}
\tilde{\lambda} - \frac{k^2 q}{2 \lambda} \sim - \frac{k^4 q^2}{8 \lambda^3} ~~ \hbox{ as } \lambda \rightarrow \infty,
\end{equation*}
the term $R_{22}(\lambda)$ is clearly asymptotically of smaller order than $\lambda^{-\frac32}e^{\lambda \Re g(t_0)}$.
\n It remains to analyze the term $R_{1}(\lambda)$ defined in \eqref{I0 R1}. First, let us write

\begin{equation*}
h_\lambda(t) = e^{i \tilde{\lambda} x_2(t)} - 1 = i \tilde{\lambda} x_2(t) + \tilde{h}_\lambda(t),
\end{equation*}

\n so that

\begin{equation*}
R_{1}(\lambda) = i \tilde{\lambda} \int_{\PC} f x_2 e^{\lambda g} dt + \int_{\PC} f e^{\lambda g} \tilde{h}_\lambda dt.
\end{equation*}

\n The second integral is much smaller than $\lambda^{-\frac{3}{2}} e^{\lambda \Re g(t_0)}$, because analogously to \eqref{h bound} we can estimate $|\tilde{h}_\lambda|$ by $1/ \lambda^2$ uniformly for $t \in \PC$. So is the first integral, because $\tilde{\lambda}$ is of order $1 / \lambda$ and $f(t_0) x_2(t_0) = 0$ and therefore, for some constant $c \in \CC$, 

\begin{equation*}
\tilde{\lambda}\int_{\PC} f x_2 e^{\lambda g} dt = \frac{1}{\lambda^{\frac{5}{2}}} e^{\lambda g(t_0)} \left[ c + O\left( \tfrac{1}{\lambda} \right)\right]  ~~\hbox {as } \lambda \to \infty.
\end{equation*}

\n Thus, we obtain

\begin{equation*}
\lambda I(\lambda) \sim \frac{C_2}{\lambda^{\frac{3}{2}}} e^{\lambda g(t_0)},
\end{equation*}

\n which concludes the proof.

\section{Proof of Theorem~\ref{THM corner}} \label{SECT corner proof}
\setcounter{equation}{0}

As already discussed in Section~\ref{SEC corner}, it suffices to analyze the (sum of the) integrals $I_1$ and $I_2$, which are taken over the line segments $L_1$ and $L_2$, respectively. Let us focus on $I_1$; the result for $I_2$ will be analogous. We parametrize the contour $-L_1$ (the contour $L_1$ traversed backwards) by $x_1(t)=t$ and $x_2(t) = - m t$ for $t \in [a_1, 0]$, where $m = \tan \theta$. Thus, $g(t) =(1 -i m) t$. Proceeding in the same way as in the derivation of formula~\eqref{I 1}, by integrating by parts twice, we find that

\begin{equation*}
\begin{split}
-I_1(\lambda) = &\int_{a_1}^0 \left[ (x_2', -x_1') \cdot V +  i \lambda g' v + i \tilde{\lambda} x_1' v  \right] e^{\lambda g + i\tilde{\lambda} x_2} dt
= \frac{1}{\lambda} \int_{a_1}^0 f  e^{\lambda g + i\tilde{\lambda} x_2} dt +
\\[.2in]
&+  \frac{1}{\lambda} \int_{a_1}^0 \left[ \left(\lambda \tilde{\lambda} - \tfrac{k^2 q}{2} \right) (x_2' + ix_1') v + i\tilde{\lambda} x_2' V \cdot (i,1)  \right] e^{\lambda g + i\tilde{\lambda} x_2} dt + 
\\[.2in]
&+\left. \left( i v -  \frac{1}{\lambda} V \cdot (i,1)  \right) e^{\lambda g + i \tilde{\lambda} x_2} \right|_{a_1}^0. 
\end{split}
\end{equation*}
Note that in contrast to \eqref{I 1} this formula contains boundary terms, since the contour is not closed.
We recall that $v, V$ are defined by \eqref{v V}, $\tilde{\lambda}$ is given by \eqref{lambda tild asymp} and

\begin{equation*}
f = \frac{k^2 q}{2} (x_2' + ix_1') v + V' \cdot (i,1).
\end{equation*}

\n Separating out the $t=0$ contribution from the boundary term above, we can now write

\begin{equation} \label{I_1}
-I_1(\lambda) = i v(0) -  \frac{1}{\lambda} V(0) \cdot (i,1) + \frac{1}{\lambda} \int_{a_1}^0 f  e^{\lambda g} dt + R_1(\lambda),
\end{equation}

\n where

\begin{equation*}
\begin{split}
R_1(\lambda) = &-\left( i v(a_1) -  \frac{1}{\lambda} V(a_1) \cdot (i,1)  \right) e^{\lambda g(a_1) + i \tilde{\lambda} x_2(a_1)} + \frac{1}{\lambda} \int_{a_1}^0 f  e^{\lambda g} \left[ e^{i\tilde{\lambda} x_2} - 1 \right] dt +
\\[.2in]
&+ \frac{1}{\lambda} \int_{a_1}^0 \left[ \left(\lambda \tilde{\lambda} - \tfrac{k^2 q}{2} \right) (x_2' + ix_1') v + i\tilde{\lambda} x_2' V \cdot (i,1)  \right] e^{\lambda g + i\tilde{\lambda} x_2} dt.
\end{split}
\end{equation*}

\n It is not hard to show that $R_1(\lambda) = O \left( \lambda^{-3} \right)$, as $\lambda \to 0$. Indeed, the first term is exponentially small, as $\Re g(a_1)<0$. The second term can be bounded using \eqref{h bound}. Namely, there exists $C>0$, such that 

\begin{equation*}
\left| \frac{1}{\lambda} \int_{a_1}^0 f  e^{\lambda g} \left[ e^{i\tilde{\lambda} x_2} - 1 \right] dt \right| \leq \frac{C}{\lambda^2} \int_{a_1}^0 e^{\lambda t} dt = C \frac{1 - e^{\lambda a_1}}{\lambda^3} \leq \frac{C}{\lambda^3}~,
\end{equation*}

\n where in the last step we used that $a_1<0$. The third term in $R_1$ can be bounded analogously. Now, since $\Re g(t)$ attains its unique maximum over the interval $[a_1,0]$ at the boundary point $t=0$, we can apply Laplace's method to obtain the asymptotic formula \cite{olver}

\begin{equation*}
\int_{a_1}^0 f  e^{\lambda g} dt = e^{\lambda g(0)} \left[ \frac{f(0)}{g'(0) \lambda} + O \left( \lambda^{-2} \right) \right].
\end{equation*}

\n We insert this in \eqref{I_1} and simplify the resulting expression. In view of 

\begin{equation*}
V'(0) \cdot (i,1) = i u^{\text{in}}_{x_1 x_1}(0) + (1-im) u^{\text{in}}_{x_1x_2}(0) - m u^{\text{in}}_{x_2x_2}(0),
\end{equation*}

\n the resulting asymptotic formula reads

\begin{equation*}
-I_1(\lambda) = i u^{\text{in}}(0) -  \frac{1}{\lambda} \nabla u^{\text{in}}(0) \cdot (i,1) + \frac{c_1}{\lambda^2}  + O \left( \lambda^{-3} \right),
\end{equation*}

\n with

\begin{equation*}
c_1 = \frac{1}{1-im} \left[ \frac{k^2q}{2} (i-m) u^{\text{in}}(0) + i u^{\text{in}}_{x_1 x_1}(0) + (1-im) u^{\text{in}}_{x_1x_2}(0) - m u^{\text{in}}_{x_2x_2}(0) \right].
\end{equation*}

Turning to $I_2$, the line segment $L_2$ is parametrized by $x_1(t) = t$ and $x_2(t) = m t$ for $t \in [a_2,0]$. Therefore, analogously, we have 

\begin{equation*}
I_2(\lambda) = i u^{\text{in}}(0) -  \frac{1}{\lambda} \nabla u^{\text{in}}(0) \cdot (i,1) + \frac{c_2}{\lambda^2}  + O \left( \lambda^{-3} \right),
\end{equation*}

\n where $c_2$ is obtained from $c_1$ by negating $m$:

\begin{equation*}
c_2 = \frac{1}{1+im} \left[ \frac{k^2q}{2} (i+m) u^{\text{in}}(0) + i u^{\text{in}}_{x_1 x_1}(0) + (1+im) u^{\text{in}}_{x_1x_2}(0) + m u^{\text{in}}_{x_2x_2}(0) \right].
\end{equation*}

\n Thus,

\begin{equation*}
I_1(\lambda) + I_2(\lambda) =  \frac{c_2 - c_1}{\lambda^2}  + O \left( \lambda^{-3} \right).
\end{equation*}

\n A direct calculation and simplification shows that

\begin{equation*}
c_2 - c_1 = \frac{2m}{1+m^2} \left[ k^2 q u^{\text{in}}(0) + \Delta u^{\text{in}}(0) \right].
\end{equation*}

\n To conclude the proof, it remains to use the Helmholtz equation satisfied by the incident wave and substitute $\Delta u^{\text{in}}(0) = -k^2 u^{\text{in}}(0)$. 


\section{Appendix: The Pompeiu and Schiffer problems}
\label{appendix}
\setcounter{equation}{0}

In this appendix we give a brief overview of the Pompeiu and Schiffer problems and relate both to the nonscattering problem. In doing so, we  significantly rely on \cite{bern80} and \cite{GS91} (see also \cite{volch}). Throughout, let $D \subset \RR^d$, with $d \geq 2$, be a nonempty, open, bounded set, and let $\Sigma$ denote the group of rigid motions of $\RR^d$ onto itself:

\begin{equation*}
\Sigma =\left\{ \sigma :\RR^d \to \RR^d: \sigma(x) = Ax + b, \ \text{where} \ A \in SO(d) \ \text{and} \ b \in \RR^d \right\}
\end{equation*}

\begin{definition}
We say that $D$ has the Pompeiu property, iff the only $f\in C(\RR^d)$ satisfying

\begin{equation*}
\int_{\sigma(D)} f(x) dx = 0, \qquad \qquad \forall \ \sigma \in \Sigma
\end{equation*}

\n is the function $f=0$.

\end{definition}

The Pompeiu problem \cite{pomp} asks for a characterization of the sets $D$ that possess the Pompeiu property. Although considerable research has been devoted to this problem, \cite{zalc}, it remains open: no explicit characterization — {\it e.g.}, a geometric one — is known, even in the case $d=2$. A key result of Brown, Schreiber, and Taylor \cite{BST73} provides an implicit characterization of the Pompeiu property in terms of the complexified Fourier transform

\begin{equation}
\hat{\chi}_D(\xi) = \int_D e^{i \xi \cdot x} dx,
\end{equation}

\n which extends as an entire function of $\xi = (\xi_1,...,\xi_d) \in \CC^d$. Namely: $D$ fails to have the Pompeiu property iff there exists $\rho \in \CC \backslash \{0\}$, such that

\begin{equation} \label{FT vanish on M}
\hat{\chi}_D = 0 \qquad \text{on} \quad \CM_\rho =\{\xi \in \CC^d: \xi \cdot \xi = \rho\}.
\end{equation}

\n The authors first show that the failure of the Pompeiu property is equivalent to the existence of a common zero $\xi_0$ for the family of functions $\left\{\hat{\chi}_{\sigma(D)} : \sigma \in \Sigma\right\}$. Since $\hat{\chi}_{D}(0)$ equals the measure of $D$, it follows that $\xi_0 \neq 0$. Now, $\hat{\chi}_{\sigma(D)}$ vanishes at $\xi_0$ for all $\sigma \in \Sigma$, and therefore $\hat{\chi}_D$ vanishes on the ``complexified sphere" $\CM_\rho$, where $\rho = \xi_0 \cdot \xi_0$. Indeed, by a change of variables in the Fourier integral, for any $\xi \in \CM_\rho$, there exists a rotation $\sigma \in \Sigma$, such that $\hat{\chi}_D(\xi) = \hat{\chi}_{\sigma(D)}(\xi_0)$. We note that Fourier transform techniques, in the context of the Pompeiu problem, were first used by Zalcman \cite{zalc}.

Pompeiu \cite{pomp} and Christov \cite{chris} showed that a square in $\RR^2$ has the Pompeiu property. In contrast, Tchakaloff \cite{Chak44} observed that the disk does not have the Pompeiu property. Indeed, let $B_r$ denote the disk of radius $r$ centered at the origin. One can compute the Fourier transform explicitly \cite{zalc}:

\begin{equation} \label{J_1 disk}
\hat{\chi}_{B_r}(\xi) = 2\pi r \frac{J_1 \left( r \sqrt{\xi_1^2 + \xi_2^2} \right)}{\sqrt{\xi_1^2 + \xi_2^2}}, \qquad \qquad \xi \in \CC^2,
\end{equation}

\n where $J_1$ denotes the Bessel function of order 1. It is now evident that if $\rho > 0$ is such that $r\sqrt{\rho}$ is a zero of $J_1$, then $\hat{\chi}_{B_r} = 0$ on $\CM_\rho$, and hence (according to the characterization described above) $B_r$ does not have the Pompeiu property. On the other hand, it is shown in \cite{BST73} that any elliptical region $E = \left\{x : \frac{x_1^2}{a^2} + \frac{x_2^2}{b^2} < 1\right\}$, with $a>b>0$, has the Pompeiu property. This follows from the formula

\begin{equation*}
\hat{\chi}_{E}(\xi) = 2\pi ab \frac{J_1 \left( \sqrt{a^2 \xi_1^2 + b^2 \xi_2^2} \right)}{\sqrt{a^2\xi_1^2 + b^2\xi_2^2}}, \qquad \qquad \xi \in \CC^2.
\end{equation*}

\n Indeed, since $J_1$ has only real zeros, the expression above can vanish on $\CM_\rho$ only if the argument of $J_1$ remains real valued there, which occurs only when $a=b$ and $\rho>0$. By studying the asymptotic behavior of the Fourier transform, it was also shown in \cite{BST73} that every convex set with at least one true corner has the Pompeiu property. Garofalo and Segala \cite{GS91} (see also \cite{GS91TR, GS94} and the references therein) studied domains with smooth boundaries using the characterization \eqref{FT vanish on M}, combined with the asymptotic analysis of the Fourier transform via the method of steepest descent. Their approach follows that of Bernstein \cite{bern80}, who used the method of stationary phase to study the asymptotics of the Fourier transform in the context of the Schiffer problem (see below).

The celebrated conjecture — also referred to as the Pompeiu conjecture — asserts that, modulo sets of measure zero, the disk is the only simply connected domain in $\RR^2$ that does not have the Pompeiu property. More generally, if $\partial D$ is homeomorphic to the unit sphere in $\RR^d$, then it is conjectured that $D$ has the Pompeiu property iff it is not a ball. This conjecture is closely related to a symmetry problem in partial differential equations known as Schiffer's problem. To describe the latter, for the remainder of this section, let us additionally assume that $D$ is a Lipschitz domain. Williams \cite{w1} showed the following equivalence, attributing it to Brown, Schreiber, and Taylor: a simply connected, bounded  Lipschitz domain, $D$, does not have the Pompeiu property iff the following overdetermined boundary value problem

\begin{equation} \label{u Schif}
\begin{cases}
-\Delta u = \rho u, \qquad &\text{in} \ D
\\
u = 1, &\text{on} \ \partial D
\\
\partial_n u = 0, &\text{on} \ \partial D
\end{cases}
\end{equation}

\n has a solution for some $\rho > 0$. The proof proceeds as follows. Suppose that $D$ does not have the Pompeiu property. Then $\hat{\chi}_D = 0$ on $\CM_\rho$ for some $0 \neq \rho \in \CC$. Since $\xi \cdot \xi - \rho = \xi_1^2 + ... +\xi_d^2 - \rho$ is an irreducible polynomial in $\CC^d$, we conclude that the entire function $\hat{\chi}_D$ must be divisible by it, {\it i.e.},

\begin{equation*}
T(\xi) = \frac{\hat{\chi}_D(\xi)}{\xi_1^2 + ... \xi_d^2 - \rho}
\end{equation*}

\n is an entire function in $\CC^d$. Therefore (cf. \cite{hor}) we may write $T = \hat{\psi}$, where $\psi$ is a distribution with compact support in $\RR^d$. Cross multiplying the above equation and inverting the Fourier transform we arrive at

\begin{equation*}
\Delta \psi + \rho \psi = - \chi_D
\end{equation*}

\n Since $\psi$ has compact support and is real analytic outside of $D$, it must vanish there: $\psi = 0$ in $\RR^d \backslash \overline{D}$. Consequently, the restriction $\psi \big |_D$ satisfies

\begin{equation} \label{BVP Schif}
\begin{cases}
\Delta \psi + \rho \psi = -1, \qquad &\text{in} \ D
\\
\psi = \partial_n \psi = 0, &\text{on} \ \partial D.
\end{cases}
\end{equation}

\n It remains to set $u = \rho \psi + 1$. The conclusion that $\rho>0$ follows from the fact that $u$ is a Neumann eigenfunction of $-\Delta$. The converse implication follows analogously.

We say that $D$ has the Schiffer property if \eqref{u Schif} has no solution for any $\rho>0$. For bounded, simply connected Lipschitz domains having the Pompeiu property or the Schiffer property is equivalent. Using tools from free boundary regularity theory, Williams \cite{w2} showed that if $\partial D$ is not a real analytic hypersurface, then $D$ has the Schiffer property. Bernstein \cite{bern80} (see also \cite{vog94}) showed that if \eqref{u Schif} has a solution for infinitely many values of $\rho$, then $D$ is a disk (provided $D \subset \RR^2$ is simply connected and bounded with a $C^2$ boundary). This result was later extended to any number of dimensions by Bernstein and Yang \cite{BY}.

The boundary value problem arising in the nonscattering problem, and given by \eqref{u system}, is related to the overdetermined problem \eqref{BVP Schif}. Indeed, if in \eqref{u system} we replace the incident wave $u^{\textbf{in}}$ by a constant function $1$, then the function $\psi = u/k^2(q-1)$ solves \eqref{BVP Schif} with $\rho=k^2q$. Of course, the incident wave $u^{\textbf{in}}$ cannot actually be a constant, as it must satisfy the Helmholtz equation with nonzero wave number $k$. 


\section*{Acknowledgements}
{This research was partially supported by NSF Grant DMS-22-05912.}

\bibliographystyle{abbrv}
\bibliography{ref}

\begin{thebibliography}{10}

\bibitem{MAPL}
\url{https://sites.math.rutgers.edu/~vogelius/software.html}.

\bibitem{bern80}
C.~A. Berenstein.
\newblock An inverse spectral theorem and its relation to the {P}ompeiu
  problem.
\newblock {\em J. Analyse Math.}, 37:128--144, 1980.

\bibitem{BY}
C.~A. Berenstein and P.~Yang.
\newblock An overdetermined {N}eumann problem in the unit disk.
\newblock {\em Adv. in Math.}, 44(1):1--17, 1982.

\bibitem{BPS}
E.~Bl{\aa}sten, L.~P\"aiv\"arinta, and J.~Sylvester.
\newblock Corners always scatter.
\newblock {\em Comm. Math. Phys.}, 331(2):725--753, 2014.

\bibitem{BST73}
L.~Brown, B.~M. Schreiber, and B.~A. Taylor.
\newblock Spectral synthesis and the {P}ompeiu problem.
\newblock {\em Ann. Inst. Fourier (Grenoble)}, 23(3):125--154, 1973.

\bibitem{Buk}
A.~L. Bukhgeim.
\newblock Recovering a potential from cauchy data in the two-dimensional case.
\newblock {\em Journal of Inverse and Ill-posed Problems}, 16(1):19--33, 2008.

\bibitem{CGH}
F.~Cakoni, D.~Gintides, and H.~Haddar.
\newblock The existence of an infinite discrete set of transmission
  eigenvalues.
\newblock {\em SIAM J. Math. Anal.}, 42:237--255, 2010.

\bibitem{CV}
F.~Cakoni and M.~S. Vogelius.
\newblock Singularities almost always scatter: Regularity results for
  non-scattering inhomogeneities.
\newblock {\em Comm. on Pure \& Applied Math.}, 76:4022--4047, 2023.

\bibitem{chris}
C.~Christov.
\newblock Sur un probl\`eme de {M}. {P}ompeiu.
\newblock {\em Mathematica, Timi\c soara}, 23:103--107, 1948.

\bibitem{CK}
D.~Colton and R.~Kress.
\newblock {\em Inverse acoustic and electromagnetic scattering theory},
  volume~93 of {\em Applied Mathematical Sciences}.
\newblock Springer-Verlag, Berlin, fourth edition, 2019.

\bibitem{CK19}
D.~Colton and R.~Kress.
\newblock {\em Inverse acoustic and electromagnetic scattering theory},
  volume~93 of {\em Applied Mathematical Sciences}.
\newblock Springer, Cham, fourth edition, [2019] \copyright 2019.

\bibitem{EH}
J.~Elschner and G.~Hu.
\newblock Acoustic scattering from corners, edges and circular cones.
\newblock {\em Arch. Rational Mech. Anal.}, 228(2):653--690, 2018.

\bibitem{GS91}
N.~Garofalo and F.~Seg\`ala.
\newblock Asymptotic expansions for a class of {F}ourier integrals and
  applications to the {P}ompeiu problem.
\newblock {\em J. Analyse Math.}, 56:1--28, 1991.

\bibitem{GS91TR}
N.~Garofalo and F.~Seg\`ala.
\newblock New results on the {P}ompeiu problem.
\newblock {\em Trans. Amer. Math. Soc.}, 325(1):273--286, 1991.

\bibitem{GS94}
N.~Garofalo and F.~Seg\`ala.
\newblock Univalent functions and the {P}ompeiu problem.
\newblock {\em Trans. Amer. Math. Soc.}, 346(1):137--146, 1994.

\bibitem{hor}
L.~H\"ormander.
\newblock {\em The analysis of linear partial differential operators. {I}}.
\newblock Classics in Mathematics. Springer-Verlag, Berlin, 2003.
\newblock Distribution theory and Fourier analysis, Reprint of the second
  (1990) edition [Springer, Berlin; MR1065993 (91m:35001a)].

\bibitem{SS24}
P.-Z. Kow, S.~Larson, M.~Salo, and H.~Shahgholian.
\newblock Quadrature domains for the {H}elmholtz equation with applications to
  non-scattering phenomena.
\newblock {\em Potential Anal.}, 60(1):387--424, 2024.

\bibitem{LHY}
L.~Li, G.~Hu, and J.~Yang.
\newblock Piecewise-analytic interfaces with weakly singular points of
  arbitrary order always scatter.
\newblock {\em J. Func. Anal.}, 284:109800, 2023.

\bibitem{olver}
F.~W.~J. Olver.
\newblock {\em Asymptotics and special functions}.
\newblock AKP Classics. A K Peters, Ltd., Wellesley, MA, 1997.
\newblock Reprint of the 1974 original [Academic Press, New York; MR0435697 (55
  \#8655)].

\bibitem{PSV}
L.~P\"aivarinta, M.~Salo, and E.~V. Vesalainen.
\newblock Strictly convex corners scatter.
\newblock {\em Rev. Mat. Iberoam.}, 33:1369--1396, 2017.

\bibitem{pomp}
D.~Pompeiu.
\newblock Sur certains systèmes d'équations linéaires et sur une propriété
  intégrale des fonctions de plusieurs variables.
\newblock {\em C. R. Acad. Sci. Paris}, 188:1138--1139, 1929.

\bibitem{SS21}
M.~Salo and H.~Shahgholian.
\newblock Free boundary methods and non-scattering phenomena.
\newblock {\em Res. Math. Sci.}, 8(4):Paper No. 58, 19, 2021.

\bibitem{SS25}
M.~Salo and H.~Shahgholian.
\newblock A free boundary approach to non-scattering obstacles with vanishing
  contrast.
\newblock 2025. arXiv 2506.22328.

\bibitem{Chak44}
L.~Tchakaloff.
\newblock Sur un probl\`eme de {D}. {P}omp\'eiu.
\newblock {\em Annuaire [Godi\v snik] Univ. Sofia. Fac. Phys.-Math. Livre 1.},
  40:1--14, 1944.

\bibitem{vog94}
M.~Vogelius.
\newblock An inverse problem for the equation {$\Delta u=-cu-d$}.
\newblock {\em Ann. Inst. Fourier (Grenoble)}, 44(4):1181--1209, 1994.

\bibitem{VX}
M.~S. Vogelius and J.~Xiao.
\newblock Finiteness results concerning nonscattering wave numbers for incident
  plane and {H}erglotz waves.
\newblock {\em SIAM J. Math. Anal.}, 53(5):5436--5464, 2021.

\bibitem{VX2}
M.~S. Vogelius and J.~Xiao.
\newblock Finiteness results for non-scattering {H}erglotz waves. the case of
  inhomogeneities obtained by very general perturbation of disks. to appear.
\newblock {\em SIAM J. Math. Anal.}, 2025.

\bibitem{volch}
V.~V. Volchkov.
\newblock {\em Integral geometry and convolution equations}.
\newblock Kluwer Academic Publishers, Dordrecht, 2003.

\bibitem{w1}
S.~A. Williams.
\newblock A partial solution of the {P}ompeiu problem.
\newblock {\em Math. Ann.}, 223(2):183--190, 1976.

\bibitem{w2}
S.~A. Williams.
\newblock Analyticity of the boundary for {L}ipschitz domains without the
  {P}ompeiu property.
\newblock {\em Indiana Univ. Math. J.}, 30(3):357--369, 1981.

\bibitem{zalc}
L.~Zalcman.
\newblock A bibliographic survey of the {P}ompeiu problem.
\newblock In {\em Approximation by solutions of partial differential equations
  ({H}anstholm, 1991)}, volume 365 of {\em NATO Adv. Sci. Inst. Ser. C: Math.
  Phys. Sci.}, pages 185--194. Kluwer Acad. Publ., Dordrecht, 1992.

\end{thebibliography}

\end{document}